\documentclass[english,12pt]{article}
\pdfoutput=1
\usepackage[a4paper]{geometry}
\usepackage{babel}
\usepackage{amsmath}
\usepackage{amsthm}
\usepackage{amssymb}
\usepackage{graphicx}
\usepackage{lmodern}
\usepackage{stix}
\usepackage[kerning=true]{microtype}
\usepackage[cmyk]{xcolor}
\usepackage[unicode=true,pdfusetitle,
 bookmarks=true,bookmarksnumbered=false,bookmarksopen=false,
 breaklinks=false,pdfborder={0 0 0},backref=false,colorlinks=true,
 urlcolor=blue]{hyperref}
\usepackage{doi}

\hypersetup{pdftitle={Nonlinear stability at the Eckhaus boundary},
 pdfauthor={Julien Guillod, Guido Schneider, Peter Wittwer, and Dominik Zimmermann}}

\newcommand{\R}{{\mathbb R}}
\newcommand{\ie}{\emph{i.e.}}

\newtheorem{theorem}{Theorem}[section]
\newtheorem{lemma}[theorem]{Lemma}

\theoremstyle{definition}

\newtheorem{remark}[theorem]{Remark}

\title{Nonlinear stability at the Eckhaus boundary}
\author{Julien Guillod$^1$,  Guido Schneider$^2$,\\ Peter Wittwer$^3$, Dominik Zimmermann$^2$
\\[3mm]
{\small $^{1}$  LJLL,
Sorbonne Universit\'e,
4 Place Jussieu, F-75005 Paris, France}\\
{\small $^{2}$ IADM,
Universit\"{a}t Stuttgart, Pfaffenwaldring 57, D-70569 Stuttgart, Germany }
\\
{\small $^{3}$ D\'epartement de Physique Th\'eorique,
Universit\'{e} de Gen\`{e}ve,} \\[-1mm] {\small
24 quai Ernest-Ansermet,
CH-1211 Gen\`{e}ve 4,
Switzerland}
}

\date{March 9, 2018}

\begin{document}
\maketitle

\begin{abstract}
The real Ginzburg-Landau equation possesses a family of spatially periodic equilibria.
If the wave number of an equilibrium is strictly below the 
so called Eckhaus boundary the equilibrium is  known 
to be spectrally and diffusively stable, \ie, stable w.r.t.~small spatially localized 
perturbations. If the wave number is above the  Eckhaus boundary 
the equilibrium is unstable. Exactly at the boundary  spectral stability holds.
The purpose of the present paper is to establish the diffusive stability of these equilibria. 
The limit profile is determined by 
a nonlinear equation since a nonlinear term turns out to be marginal
w.r.t.~the linearized dynamics.
\end{abstract}

\section{Introduction}

The Ginzburg-Landau equation
\begin{equation} \label{rGL}
\partial_T A = \partial^2_X A + A - A \left| A \right|^2,
\end{equation}
with $ T \geq 0 $, $ X \in \mathbb{R} $, and $ A(X,T) \in \mathbb{C} $ 
appears as a universal amplitude equation for the description of a number of pattern forming systems close to the first instability, cf. \cite{NW69}. See \cite{SZ13,SUbook} for a recent 
overview about the  mathematical justification of the so called Ginzburg-Landau approximation.
The  stationary solutions of \eqref{rGL}, namely
\begin{equation} \label{vms1}
A_q = \sqrt{1-q^2} e^{iqX} 
\end{equation}
are known  to be spectrally 
stable for $q^2 \leq 1/3$ and 
unstable for $q^2 > 1/3$. This was observed first  in \cite{Eckhaus65book}
and therefore $q^2 = 1/3$ is called the 
Eckhaus  or sideband stability boundary.  

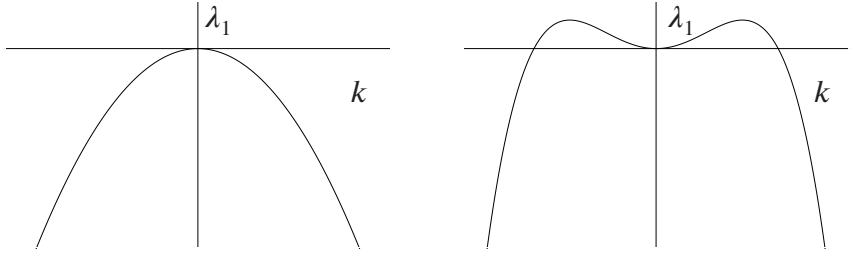
\begin{figure}[htbp]
   \centering
   \includegraphics[width=2in]{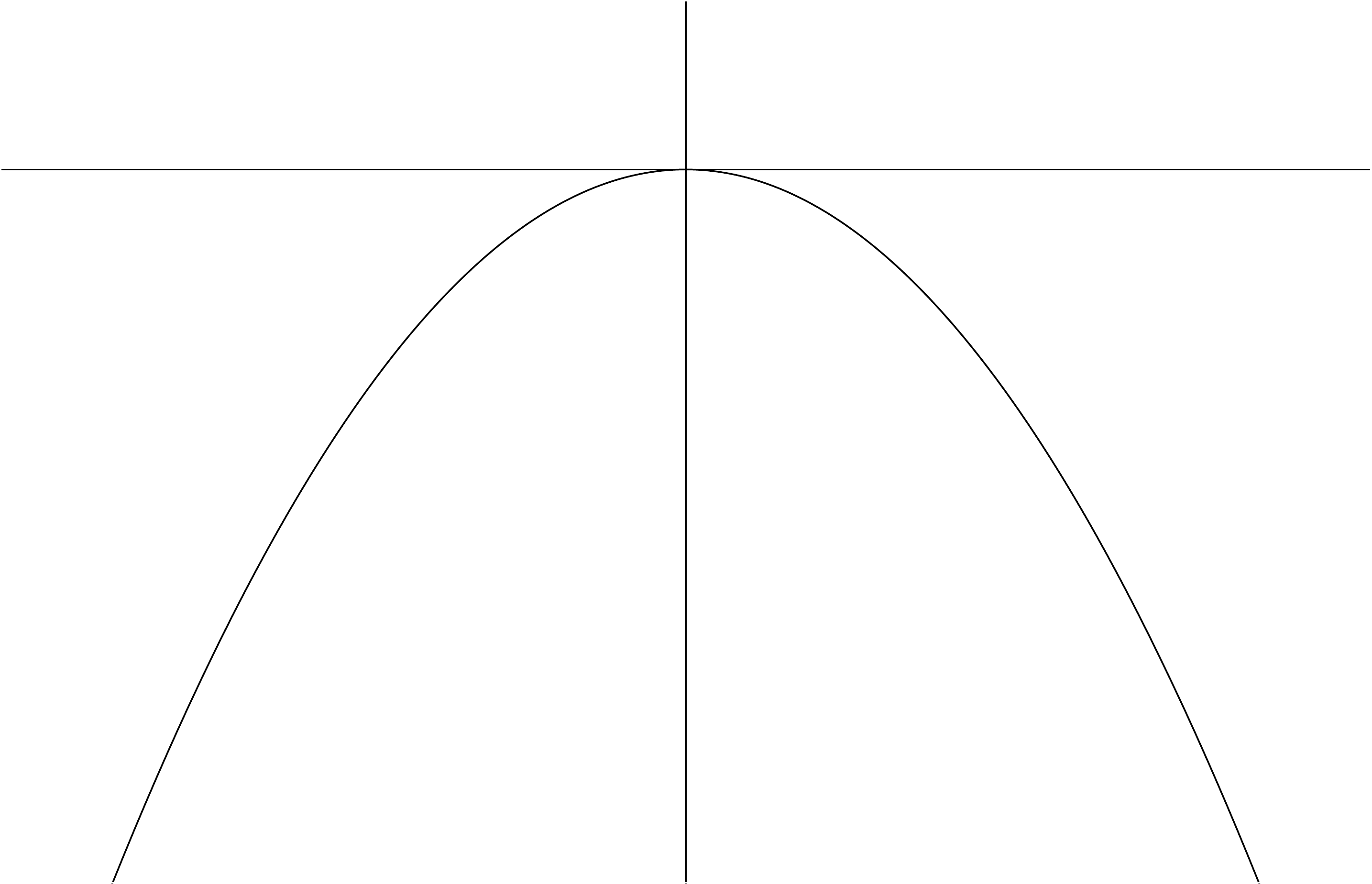} \qquad
   \includegraphics[width=2in]{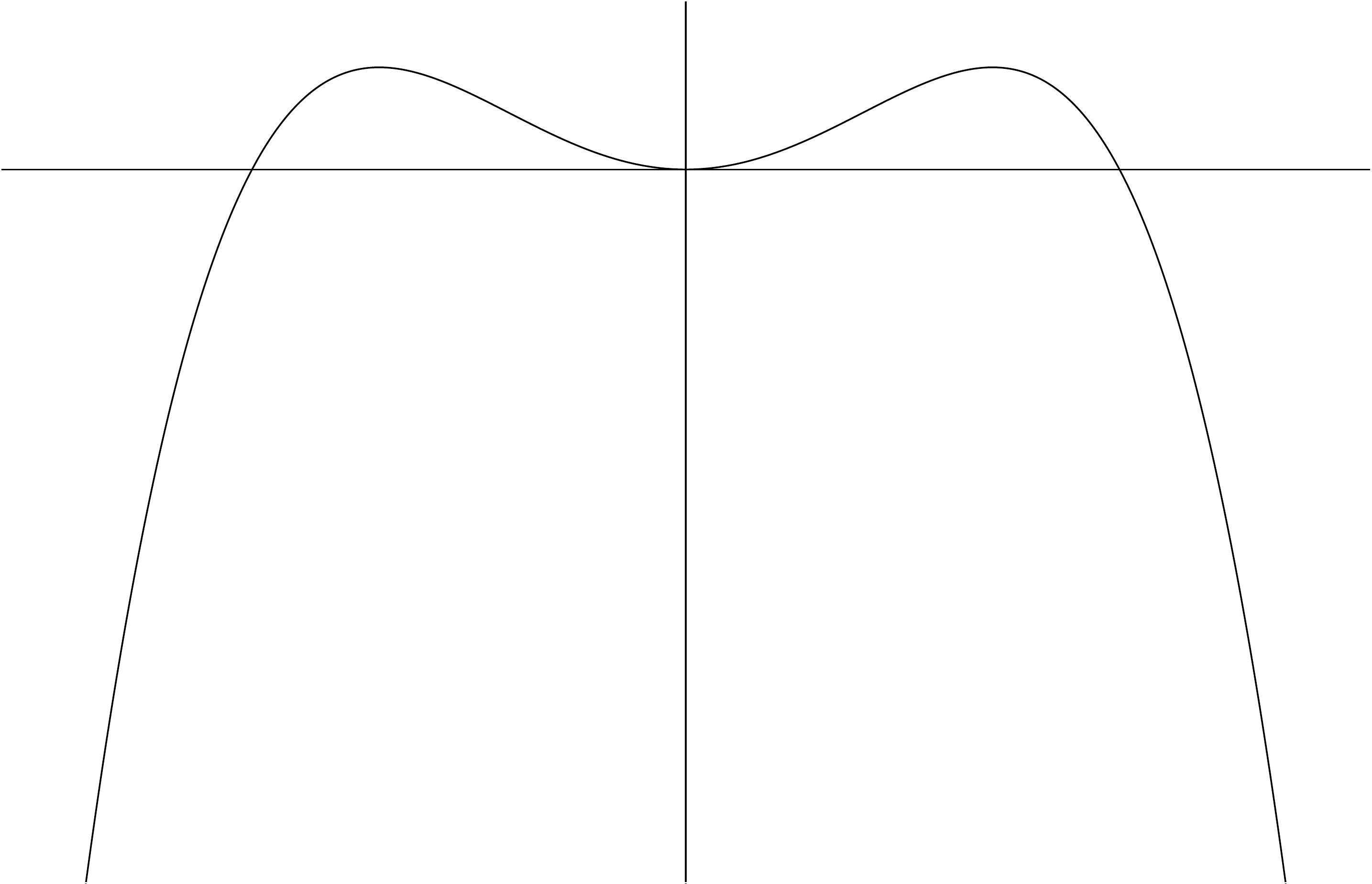} 
\setlength{\unitlength}{.4in}
\vspace*{-4cm}

\begin{picture}(7,5)(0,0)
\put(0.5,3.98){$\lambda_1$}
\put(6.5,3.98){$\lambda_1$}
\put(2.4,3){$k $}
\put(8.4,3){$k $}
\end{picture}
   \caption{The linearization around the equilibrium is solved by $ e^{ikx+\lambda_{1,2} t} V_{1,2} $ with $ V_{1,2}  \in \mathbb{C}^2 $ and  $ \lambda_1(k) > \lambda_2(k) $. The left panel 
   shows the curve $ k \mapsto \lambda_1(k) $ in the stable case and the right panel 
   in the unstable case.}
   \label{figure1}
\end{figure}

It took more than twenty years to establish the diffusive stability of  the
spectrally stable equilibria, \ie, the stability 
w.r.t.~small spatially localized 
perturbations. In \cite{CEE92} this result has been shown by using  $ L^1 $-$ L^{\infty} $ estimates 
and in \cite{BK92} 
by using a renormalization group approach 
to additionally establish the exact asymptotic decay of the perturbation in time. 
The proofs are based on the fact that 
the nonlinear terms are irrelevant w.r.t.~the linear diffusion
$$  
\varphi(X,T) = \frac{\varphi^*}{\sqrt{T}}e^{- \frac{X^2}{4T}} + \mathcal{O}\!\left(\frac{1}{T}\right),
$$
and so 
the renormalized perturbation converges  towards a Gaussian limit. 

In contrast to exponential decay rates, polynomial decay rates occurring in diffusion 
do not allow in general to control all nonlinear terms in a neighborhood of the origin. 
The nonlinear terms can be divided 
into irrelevant ones  which show faster decay rates than the  linear diffusion terms
$ \partial_T \varphi $ and $ \partial_X^2  \varphi $, into marginal ones which show the same decay 
rates and into the ones which decay slower and which would lead to a completely different asymptotic 
behavior for $ T \to \infty $. Linear diffusive  behavior exhibits  the following asymptotic
decay rates 
$$
\varphi \sim T^{-1/2}, \quad \partial_X
\sim T^{-1/2}, \quad \text{and} \quad \partial_T \sim T^{-1}
$$ 
and so in a nonlinear diffusion equation 
$$ 
\partial_T \varphi - \partial_X^2  \varphi = \varphi^{p_0}  (\partial_X  \varphi)^{p_1} (\partial_X^2  \varphi)^{p_2}
$$ 
the terms on the left hand side both exhibit a decay rate  $ T^{-3/2} $, whereas the right hand side 
decays as $ T^{-(p_0 + 2 p_1 + 3 p_2)/2} $. More precisely, a term $ \varphi^{2} $ cannot be controlled by diffusion,
a term $ - \varphi^3 $ leads to a faster decay,  a term $ + \varphi^3 $  to a logarithmic growth, and a Burgers term 
$ \varphi \partial_X \varphi $ is not changing the decay rates, but the limit profile from a  Gaussian into 
a perturbed Gaussian. All other terms, satisfying $ p_0 + 2 p_1 + 3 p_2 \geq 4 $, can be controlled asymptotically by 
the left hand side.
In order to  prove that a smooth nonlinearity $ H = H(\varphi, \partial_X  \varphi,\partial_X^2  \varphi) $ can be controlled 
by diffusion we have to show that the coefficients in front of $ \varphi^2 $ and $ \varphi^3 $ vanish.
This idea has been generalized to very general systems where $ \partial_X^2 $  has been replaced by 
operators which possess a curve of eigenvalues with 
a parabolic profile for $ k \to 0 $ in Fourier or Bloch space. 
In many such systems the 
nonlinear terms turn out be irrelevant 
\cite{Schn98ARMA,DSSS,SSSU,zumbrum}.

Exactly at the Eckhaus stability boundary, $ q^2 = 1/3 $, spectral stability still holds, but 
only with $ \lambda_1 \sim - k^4 $ instead of 
$ \lambda_1 \sim - k^2 $
for $ k \to 0 $ as shown in Figure \ref{figure1}.
Therefore, we only have  the much slower asymptotic decay rates 
$$
\varphi \sim T^{-1/4}, \quad \partial_X
\sim T^{-1/4}, \quad \text{and} \quad \partial_T \sim T^{-1}.
$$ 
Due to this slow decay there is a  nonlinear term which is marginal w.r.t.~the linear  dynamics.
We find an effective equation of the form
\begin{equation} \label{show}
\partial_T \varphi  + \nu_1 \partial^4_X \varphi = \nu_2  \partial_X \left(
  (\partial_X \varphi)^2\right) + g_1,
\end{equation}
with coefficients $ \nu_1 > 0 $ and $ \nu_2 < 0 $. 
The first term on the right hand side decays as  $ \sim T^{-5/4} $ like the linear ones on the left hand side.
In $ g_1 $ we collect all terms with faster decay rates. 
Fortunately, it turns out that $ \nu_2  \partial_X \left(
  (\partial_X \varphi)^2\right) $ is not changing the decay rates, but only leads to 
a  nonlinear correction of the limit profile like the Burgers term $ \varphi \partial_X \varphi $  for diffusion \cite{BKL94}.
Our result is therefore as follows.
\begin{theorem} \label{yuri1}
For all $C > 0$, there exists $\delta > 0$ such that for any
$ \widehat{V}_0 \in L^1 \cap L^\infty$ satisfying
$\|\widehat{V}_0\|_{L^1}+\|\widehat{V}_0\|_{L^{\infty}} \leq \delta $,
the solution $A = A_{\sqrt{1/3}} + V $ of the Ginzburg-Landau equation \eqref{rGL} with $ V|_{T=0} = V_0 $ satisfies
 $$\|\widehat{V}(T)\|_{L^{\infty}} \le C \qquad   \text{and}  \qquad   \|V(T)\|_{L^{\infty}} \leq
\|\widehat{V}(T)\|_{L^1} \leq
    C (1{+}T)^{-\frac{1}{4}}$$
for all $T \ge 0$.
\end{theorem}

The proof is an adaption of the $ L^1 $-$ L^{\infty} $ scheme presented in \cite{MSU01}
to the situation of a coupling of linearly diffusive modes with linearly exponentially damped modes. 
The complications are  due to the marginal relevant nonlinear term and the slower decay rates.
We strongly believe that our result  can be transferred to general pattern forming systems, too.

The plan of the paper is as follows. 
Section \ref{sec2} recalls the formal calculations to derive \eqref{show}.
This section is not necessary for the proof of Theorem \ref{yuri1} 
but helps to understand the subsequent steps of the proof.  
The proof of Theorem \ref{yuri1} starts 
in Section \ref{secneu3} with the separation of  the linearly diffusive and the linearly exponentially damped  modes in a suitably chosen coordinate system.
In Section \ref{sec5} we  establish  the linear  decay estimates.
The formal irrelevance of the nonlinear terms can be found  in Section \ref{sec6} and 
in Appendix \ref{appB}.
The final nonlinear decay estimates can be found in Section \ref{sec7}.
For completeness the limit profile of the renormalized solution is computed in Appendix \ref{appA}.

For some of the following explicit calculations the software Mathematica \cite{mathematica} was used.

\medskip

\noindent
{\bf Notation.}
We define the Fourier transform by 
$$ 
\widehat{u}(k) = (\mathcal{F} u)(k) = \frac{1}{2 \pi} \int_{\R} u(x) e^{-ikx} \mathrm{d}x
$$ 
and  the inverse  Fourier transform by 
$$ 
u(x)  = (\mathcal{F}^{-1} \widehat{u})(x) =  \int_{\R} \widehat{u}(k)  e^{ikx} \mathrm{d}k.
$$ 
We have $ \| u \|_{\infty} \leq \| \widehat{u} \|_{1} $, where 
$ \| u \|_{\infty} = \sup_{x \in \R} | u(x) | $ is the norm in the space of bounded uniformly continuous 
functions and  $ \| u \|_{1}  = \int_{\R} |u(x) | \mathrm{d}x $ the norm in the Lebesgue space $ L^1 $.

\section{Some formal calculations}

\label{sec2}

In this section we formally derive \eqref{show}.
This section is not necessary for the proof of Theorem \ref{yuri1} 
but helps to understand the subsequent steps of the proof. 

\subsection{Equations for the deviation}

In order to obtain a semilinear system with $ X $-independent coefficients we introduce the deviation $ V $ from 
$A_q$ not in an additive, but in a multiplicative way, \ie, we set 
\begin{equation} \label{Katsun}
A(X,T)=A_{\sqrt{1/3}}(X)(1+V(X,T)) = \sqrt{\textstyle \frac{2}{3}} e^{i \sqrt{1/3} X} (1+V(X,T)).
\end{equation}
With
\begin{align*}
\partial_XA	&=A_{\sqrt{1/3}}\partial_X V+ i\sqrt{\tfrac{1}{3}}A_{\sqrt{1/3}}+i\sqrt{\tfrac{1}{3}}A_{\sqrt{1/3}}V,\\
\partial_X^2A	&=A_{\sqrt{1/3}}\partial_X^2V+2i\sqrt{\tfrac{1}{3}}\partial_XV-\tfrac{1}{3}A_{\sqrt{1/3}}-\tfrac{1}{3}A_{\sqrt{1/3}}V
\end{align*}
we find
\[
\partial_TV=\partial_X^2V+2i\sqrt{\tfrac{1}{3}}\partial_XV-\tfrac{2}{3}V-\tfrac{2}{3}\overline{V}-\tfrac{2}{3}V^2-\tfrac{4}{3}|V|^2-\tfrac{2}{3}V|V|^2.
\]
Now we split the above equation into real and imaginary part. We introduce  $V_r=\mathrm{Re}\,V$, $V_i=\mathrm{Im}\,V$ and  obtain
\begin{equation}\label{vrvi}
\begin{array}{rcl}
\partial_T V_r & = & \partial_X^2 V_r -\tfrac{4}{3} V_r -2\sqrt{\tfrac{1}{3}}\partial_X V_i -\frac{2}{3}(3 V_r^2 + V_i^2+V_r^3+V_rV_i^2) , \\
\partial_T V_i & = &  \partial_X^2 V_i+ 2\sqrt{\tfrac{1}{3}} \partial_X V_r -\frac{2}{3}(2V_rV_i+V_r^2V_i+V_i^3).
\end{array}
\end{equation}

\subsection{Spectral analysis}

The linearization around $ (V_r,V_i) =(0,0) $ is given by 
\begin{align*}
\partial_T V_r&= \partial^2_X V_r - \frac43 V_r - 2\sqrt{\tfrac{1}{3}}\partial_X
V_i , \\
\partial_T V_i &= \partial^2_X  V_i+ 2\sqrt{\tfrac{1}{3}}\partial_X V_r .
\end{align*}
It is solved by
\[
V_r = \widehat{V}_r e^{ikX} e^{\lambda T} , \quad
V_i = \widehat{V}_i e^{ikX} e^{\lambda T},
\]
where
\begin{align*}
\lambda \widehat{V}_r &= - k^2 \widehat{V}_r - \frac43 \widehat{V}_r -
2\sqrt{\tfrac{1}{3}} i k  \widehat{V}_i ,\\
\lambda \widehat{V}_i &= - k^2 \widehat{V}_i +2\sqrt{\tfrac{1}{3}} i k  \widehat{V}_r.
\end{align*}
The condition for non-trivial solutions 
\[
\det
\begin{pmatrix}
- k^2 - \lambda - \frac43 &- 2\sqrt{\tfrac{1}{3}}  ik \\
2\sqrt{\tfrac{1}{3}}  ik             & - k^2 - \lambda
\end{pmatrix}= \lambda^2 + 2 k^2 \lambda + k^4 + \frac43 \lambda = 0
\]
leads to the curves of eigenvalues
\begin{equation} \label{alcu3}
2 \lambda_{1/2}(k) = - \left( 2 k^2 + \frac43 \right) \pm
\sqrt{\left( 2 k^2 + \frac43 \right)^2 - 4 k^4}.
\end{equation}
The  expansion at $k = 0$ is given by 
\[
\lambda_ 1(k) = - \frac34 k^4 + \mathcal{O}(k^6), \quad \lambda_2(k) = - \frac43 + \mathcal{O}(k^2).
\]

\subsection{Linear asymptotic analysis}

\label{sec3}

Hence, the modes associated to the curve $ \lambda_2 $ are exponentially damped, whereas the 
curve $ \lambda_1 $ comes up to zero and leads to at most to polynomial decay rates.
For the linear equation the modes will concentrate at $ k = 0 $ such that the expansion 
of  $ \lambda_1 $ at $ k = 0 $ plays a crucial role. 
Diagonalizing the linear part leads to a change of variables $ (V_r,V_i) \mapsto (V_s,V_c) $
with the asymptotic model
$$
\partial_T \widehat{V}_c = - \frac34 k^4 \widehat{V}_c.
$$
It is solved by
\[
\widehat{V}_c(k,T) = e^{-\frac34 k^4 T} \widehat{V}_c (k,0)
\]
and shows some self-similar behavior, namely  
$$
\widehat{V}_c (\kappa T^{-\frac14},T) = e^{- \frac34 \kappa^4} \widehat{V}_c (\kappa T^{-\frac14},0) = e^{- \frac34 \kappa^4} (1 + \mathcal{O}(T^{-1/4})),
$$
provided the solution is normalized with $\widehat{V}_c (0,0)=1$.
Hence, for $ T \to \infty $ the solutions will behave like  the self-similar solution
$$
\widehat{V}_c (k,T) = \widehat{\Phi}_{lin} \bigl(k T^{\frac14} \bigr), \qquad
\text{with} \qquad \widehat{\Phi}_{lin}(\kappa) = e^{- \frac34 \kappa^4}.
$$
Transferring these formulas into physical space shows that for $ T \to \infty $  the solutions of 
\begin{equation} \label{alcu1}
\partial_T V_c = - \frac34 \partial^4_X V_c
\end{equation}
will behave like  the self-similar solution
\begin{equation} \label{alcu2}
V_c(X,T) = {T^{-1/4}} {\Phi}_{lin} \bigl({X}{T^{-1/4}}\bigr),
\end{equation}
where $ {\Phi}_{lin} =  \mathcal{F}^{-1} \widehat{\Phi}_{lin} $.

At leading order in the limit $ k\to 0$, we have
$$
\widehat{V}_r = \widehat{V}_s - \tfrac{\sqrt{3}}{2} i k \widehat{V}_c
\qquad \text{and} \qquad
\widehat{V}_i = \widehat{V}_c - \tfrac{\sqrt{3}}{2} i k \widehat{V}_s,
$$
so we expect the following scaling in the original variables
\begin{equation} \label{orig}
V_r \sim {T^{-1/2}}, \quad V_i \sim {T^{-1/4}}, \quad \partial_X
\sim{T^{-1/4}}, \quad \text{and} \quad \partial_T \sim T^{-1},
\end{equation}
at least at the linear level.

\subsection{Nonlinear asymptotic analysis}

According to the explanations from the introduction, polynomial decay rates 
do not allow to control all nonlinear terms in a neighborhood of the origin. Therefore, we have to compute the effective nonlinearity. As already said, it turns out that there is one 
marginal nonlinear term  which leads to a nonlinear correction of the linear limit profile $\Phi_{lin}$, 
but not 
to an instability or to a change in the decay rates. 

In order to compute this nonlinear correction to \eqref{alcu1} we suppose that the dynamics is in fact 
controlled by the linear dynamics \eqref{alcu2}, \ie, we consider the 
asymptotic decays given by \eqref{orig}.
Since $ V_c $ is linearly exponentially damped at $ k = 0 $, we expect that 
$ V_r $ is slaved by $ V_i $ for large times.
We find
$$
\underbrace{\partial_T V_r}_{\sim T^{-3/2}} = \underbrace{\partial_X^2 V_r}_{\sim T^{-1}} -\underbrace{\tfrac{4}{3} V_r}_{\sim T^{-1/2}}  -\underbrace{2\sqrt{\tfrac{1}{3}}\partial_X V_i}_{\sim T^{-1/2}}
- \frac{2}{3}(\underbrace{3 V_r^2}_{\sim T^{-1}} + \underbrace{V_i^2}_{\sim T^{-1/2}}+\underbrace{V_r^3}_{\sim T^{-3/2}}+\underbrace{V_rV_i^2}_{\sim T^{-1}}).
$$
Equating the terms of decay $ T^{-1/2} $ to zero yields
$$ 
0 = -\tfrac{4}{3} V_r- 2\sqrt{\tfrac{1}{3}}\partial_X V_i-\tfrac{2}{3} V_i^2 
\qquad \text{or equivalently} \qquad 
 V_r = - \tfrac{1}{2}\sqrt{3}\partial_X V_i-\tfrac{1}{2} V_i^2 .
$$ 
Inserting this into the equation for $ V_i $ 
$$
\underbrace{\partial_T V_i}_{\sim T^{-5/4}}  =  \underbrace{\partial_X^2 V_i}_{\sim T^{-3/4}}  +  \underbrace{2\sqrt{\tfrac{1}{3}}\partial_X  V_r}_{\sim T^{-3/4}}   -\tfrac{2}{3}(\underbrace{2V_rV_i}_{\sim T^{-3/4}} +\underbrace{V_r^2V_i}_{\sim T^{-5/4}} +\underbrace{V_i^3}_{\sim T^{-3/4}} )
$$
gives for the terms of decay $ T^{-3/4} $ that 
$$ 
\partial_X^2 V_i+ 2\sqrt{\tfrac{1}{3}}\partial_X ( - \tfrac{1}{2}\sqrt{3}\partial_X V_i-\tfrac{1}{2} V_i^2)
 -\tfrac{4}{3}(- \tfrac{1}{2}\sqrt{3}\partial_X V_i-\tfrac{1}{2} V_i^2 )V_i -\tfrac{2}{3}
 V_i^3 = 0 .
$$
Since this expression vanishes identically, we need to include the $ T^{-1} $ terms into the expression of $ V_r $ in terms of $ V_i $
\begin{eqnarray*}
\underbrace{\tfrac{4}{3} V_r}_{\sim T^{-1/2}} & = & \underbrace{\partial_X^2 V_r}_{\sim T^{-1}}  -\underbrace{2\sqrt{\tfrac{1}{3}}\partial_X V_i}_{\sim T^{-1/2}}
- \tfrac{2}{3}(\underbrace{3 V_r^2}_{\sim T^{-1}} + \underbrace{V_i^2}_{\sim T^{-1/2}}+\underbrace{V_rV_i^2}_{\sim T^{-1}})
\\ & = &  \underbrace{- 2\sqrt{\tfrac{1}{3}}\partial_X V_i - \tfrac{2}{3}V_i^2 }_{\sim T^{-1/2}}
\\ && + \tfrac{1}{6}
(\underbrace{-V_i^4 -15 (\partial_X V_i)^2  - 4 \sqrt{3} V_i^2 \partial_X V_i  - 6 V_i \partial_X^2 V_i -
3\sqrt{3} \partial_X^3 V_i}_{\sim T^{-1}} ).
\end{eqnarray*}
Inserting this into the equation for $ V_i $ yields
$$ 
\partial_T V_i = - \frac34 \partial^4_X V_i - \frac{3}{2}
\sqrt{3} \partial_X \left(( \partial_X V_i)^2\right) + \mathcal{O}(T^{-5/4}).
$$ 
Hence, there is a  nonlinear term which is asymptotically of the same order 
as the linear terms $ \partial_T V_i $ and $ - \frac34 \partial_X^4 V_i $ for $ T \to \infty $
and so  the asymptotic behavior will be governed by the self-similar solutions of 
\begin{equation}\label{bienne}
\partial_T V_i = - \frac34 \partial^4_X V_i -\frac{3}{2}
\sqrt{3} \partial_X \left(( \partial_X V_i)^2\right) .
\end{equation}
Fortunately, as already said,
the marginal term  $ -\frac{3}{2}
\sqrt{3} \partial_X \left(( \partial_X V_i)^2\right) $
will only lead  to a nonlinear correction of the limit profile, but not 
to an instability or to a change in the decay rates. 
Although not necessary for the proof of Theorem \ref{yuri1},
the nonlinear correction of the limit profile is computed in Appendix \ref{appA}.

\begin{remark}\label{remphasdiff}
It is not a surprise that a system of the form \eqref{bienne} is obtained.
The so-called phase diffusion equations can be derived from the Ginzburg-Landau equation for the local wave number $ \Psi $, cf. \cite{MS04MN}.
The amplitude $ \Psi $ satisfies a system of the form 
$ \partial_{\tau} \Psi = \partial_{\xi}^2 h(\Psi) $ with $ h''(0) = \lambda_1|_{k=0}[q]/2 $,
and describes  small modulations in time $ \tau = \delta T $ and space $ \xi = \delta X $ of the 
periodic wave  $ A_q $, where $ 0 < \delta \ll 1 $ is a small perturbation parameter. This equation degenerates for $ q^2 =1/3 $. 
Since at lowest order $ \Psi \sim \partial_{\xi} V_i $,
at $ q^2 = 1/3 $ we have a system
$$ 
\partial_{\tau} V_i= \partial_{\xi} (h(\partial_{\xi}  V_i)) + h.o.t. \sim   \partial_{\xi} (
(\partial_{\xi} V_i)^2) + h.o.t. 
$$
The linear term $ \partial_{\xi}^4 V_i $ is of higher order w.r.t.~the scaling used in the derivation 
of the phase diffusion equation.
\end{remark}

\section{Some preparations}
\label{secneu3}

We start now with the proof of Theorem \ref{yuri1}.

\subsection{Separation of the diffusive modes}
\label{sec4}

We introduce  $v=(V_r,V_i)^\top$ and abbreviate  \eqref{vrvi} as 
\[
\partial_Tv=Lv+N(v),
\]
where 
$$
L	=\begin{pmatrix} \partial_X^2-\tfrac{4}{3} & -2\sqrt{\tfrac{1}{3}}\partial_X \\ 2\sqrt{\tfrac{1}{3}}\partial_X & \partial_X^2 \end{pmatrix} \qquad \text{and} \qquad 
 N(v)=-\frac{2}{3}\begin{pmatrix} 3 V_r^2 + V_i^2+V_r^3+V_rV_i^2\\ 
 2V_rV_i+V_r^2V_i+V_i^3
\end{pmatrix}.
$$
At this point it turns out to be advantageous to work in Fourier space. 
Hence we consider 
\begin{equation} \label{spain}
\partial_T \widehat{v} =\widehat{L} \widehat{v} + \widehat{N}(\widehat{v}), 
\end{equation}
where $ \widehat{v} = \mathcal{F} v $, $ \widehat{L} = \mathcal{F} L \mathcal{F}^{-1} $, and  $ \widehat{N}(\widehat{v}) = \mathcal{F}( N (\mathcal{F}^{-1}\widehat{v}) ) $.

There exists a $ k_0 > 0 $ such that for all $ |k| \leq k_0 $ 
the two  curves of eigenvalues $ \lambda_{1,2} $ defined in  \eqref{alcu3} are separated, and so we define 
$$ 
 \widehat{P}_c(k)\widehat{v}(k) = \chi(k) \langle 
\widehat{\varphi}^*_{1}(k),\widehat{v}(k) \rangle \widehat{\varphi}_{1}(k),
$$ 
where 
$ \chi(k) = 1 $ for $ |k| \leq k_0/2 $,
and $ \chi(k) = 0 $ for $ |k| > k_0/2 $, and where $ \widehat{\varphi}^*_{1}(k) $ is the eigenvector 
associated to the adjoint eigenvalue problem normalized by $  \langle 
\widehat{\varphi}^*_{1}(k), \widehat{\varphi}_{1}(k)\rangle = 1 $. Moreover, define $ \widehat{P}_s(k)  \widehat{v}(k)  =  \widehat{v}(k)-  \widehat{P}_c(k)\widehat{v}(k) $. 
We use the projections to separate \eqref{spain} in two parts, namely
\begin{equation} \label{toljan}
\partial_T \widehat{v}_c  =  \widehat{L}_c \widehat{v}_c + \widehat{P}_c \widehat{N}(\widehat{v}),\qquad
\partial_T \widehat{v}_s  =  \widehat{L}_s  \widehat{v}_s + \widehat{P}_s \widehat{N}(\widehat{v}),
\end{equation}
where $ \widehat{L}_c = \widehat{L} \widehat{P}_c $ % $ \mu_c = \lambda_1 $ 
and $ \widehat{L}_s = \widehat{L} \widehat{P}_s $. By construction the operators $ \widehat{P}_s  $ and 
$ \widehat{P}_c  $  commute with  $  \widehat{L} $.
System \eqref{toljan} is solved with $ \widehat{v}_c|_{t=0} = \widehat{P}_c(k)\widehat{v}|_{t=0}  $
and $ \widehat{v}_s|_{t=0} = \widehat{P}_s(k)\widehat{v}|_{t=0}  $. Then $ \widehat{v}_c $ and 
$ \widehat{v}_s $ are defined via the solutions of  \eqref{toljan}.

Moreover, we introduce $ \widehat{V}_c $ by $  \widehat{v}_c (k,t)    =  \widehat{V}_c(k,t) \widehat{\varphi}_{1}(k) $, and 
$ \widehat{V}_s $ for  $ |k| \leq k_0/2 $ by $  \widehat{v}_s (k,t)    =  \widehat{V}_s(k,t) \widehat{\varphi}_{2}(k) $.

\subsection{Linear decay estimates}
\label{sec5}

In order to show the  nonlinear stability of $ A_{\sqrt{1/3}} $ we use the polyomial decay rates
of the linear semigroup generated by $ L $. However, the optimal  decay rate $ T^{-1/4} $ of the semigroup is only obtained 
as a mapping from $ L^1 $ to $ L^{\infty} $ in physical space,  or from 
$ L^{\infty} $ to $ L^1 $ in Fourier space. Therefore, we have to work with at least two spaces.
In Fourier space the $ L^{\infty} $-norm 
of the solutions of $ \partial_T \widehat{v} = \widehat{L} \widehat{v}  $
will be bounded and the $ L^1 $-norm will decay 
as $ T^{-1/4} $, both  for  initial conditions  in $ L^{\infty} \cap  L^1 $.

Since the sectorial operator $  \widehat{L}_s $ has spectrum in the left half plane strictly bounded away from the imaginary axis,  we obviously have the following result, cf. \cite{He81}.

\begin{lemma}
For the analytic semigroup generated by  $  \widehat{L}_s $ 
we have the estimates 
\[
 \|e^{T\widehat{L}_s}\|_{L^1\to L^1}\leq Ce^{-\sigma_s T/2}, \quad\text{ and }\quad \|e^{T\widehat{L}_s}\|_{L^\infty \to L^\infty}\leq C e^{-\sigma_s T/2},
\]
with some $\sigma_s>0$. 
\end{lemma}
For the $ \widehat{v}_c  $-part  we obtain
\begin{lemma}
Let $\nu\geq 0$.
For the analytic semigroup generated by $\widehat{L}_c$
we have the estimates 
\[
 \|e^{T\widehat{L}_c}|k|^\nu\|_{L^1\to L^1}\leq CT^{-\nu/4},\quad 
 \|e^{T\widehat{L}_c}|k|^\nu\|_{L^\infty \to L^\infty}\leq C T^{-\nu/4},\quad 
 \|e^{T\widehat{L}_c}|k|^\nu\|_{L^\infty \to L^1}\leq C T^{-(\nu+1)/4}.
\]
\end{lemma}
\noindent
{\bf Proof.}
Since $ \lambda_1(k)\leq -Ck^4$ for small $ k $ and $ e^{T \widehat{L}_c(k)}\widehat{v}_c(k) = e^{\lambda_1(k) T} \widehat{v}_c(k) $ we obviously have 
\begin{eqnarray*}
\| e^{T \widehat{L}_c}|k|^\nu\widehat{v}_c \|_{L^1} &\leq & \| e^{ \lambda_1(k) T}|k|^\nu \|_{L^{\infty}}\| \widehat{v}_c \|_{L^1}
\leq C T^{-\nu/4}\| \widehat{v}_c \|_{L^1},\\
\| e^{T \widehat{L}_c}|k|^\nu\widehat{v}_c \|_{L^\infty} &\leq & \| e^{ \lambda_1(k) T}|k|^\nu \|_{L^{\infty}}\| \widehat{v}_c \|_{L^\infty}
\leq C T^{-\nu/4}\| \widehat{v}_c \|_{L^\infty},\\
\| e^{T \widehat{L}_c}|k|^\nu\widehat{v}_c \|_{L^{1}} &\leq & \| e^{ \lambda_1(k) T}|k|^\nu \|_{L^{1}}\| \widehat{v}_c \|_{L^{\infty}}
\leq C T^{-(\nu+1)/4}\| \widehat{v}_c \|_{L^{\infty}} .
\end{eqnarray*}
\qed

\subsection{Formal irrelevance of the nonlinear terms}
\label{sec6}

After showing  decay rates for the linear semigroup we have to establish 
the irrelevance of the nonlinearity w.r.t.~this linear  behavior. In view of
future applications we will consider  a general  nonlinearity and not only quadratic and cubic terms.
In order to do so we expand the nonlinear terms into 
\begin{eqnarray*}
\widehat{P}_c
 \widehat{N}(\widehat{v}) & = & B_{2,1}(\widehat{v}_c,\widehat{v}_c) + B_{3,1}(\widehat{v}_c,\widehat{v}_c,\widehat{v}_c) 
+ B_{4,1}(\widehat{v}_c,\widehat{v}_c,\widehat{v}_c,\widehat{v}_c) + B_{5,1}(\widehat{v}_c,\widehat{v}_c,\widehat{v}_c,\widehat{v}_c,\widehat{v}_c) \\ && + B_{2,2}(\widehat{v}_c,\widehat{v}_s)+ B_{3,2}(\widehat{v}_c,\widehat{v}_c,\widehat{v}_s)
+ B_{4,2}(\widehat{v}_c,\widehat{v}_c,\widehat{v}_c,\widehat{v}_s) \\
&& + B_{2,3}(\widehat{v}_s,\widehat{v}_s) + B_{3,3}(\widehat{v}_c,\widehat{v}_s,\widehat{v}_s) 
+g_c(\widehat{v}_c,\widehat{v}_s),\\ 
\widehat{P}_s \widehat{N}(\widehat{v})& = & B_{2,4}(\widehat{v}_c,\widehat{v}_c) + B_{3,4}(\widehat{v}_c,\widehat{v}_c,\widehat{v}_c) 
+ B_{4,4}(\widehat{v}_c,\widehat{v}_c,\widehat{v}_c,\widehat{v}_c) \\ && + B_{2,5}(\widehat{v}_c,\widehat{v}_s)+ B_{3,5}(\widehat{v}_c,\widehat{v}_c,\widehat{v}_s)
 + B_{2,6}(\widehat{v}_s,\widehat{v}_s) 
+ g_s(\widehat{v}_c,\widehat{v}_s),
\end{eqnarray*}
where the $ B_{m,j} $ are symmetric $m$-linear mappings,
and where $ g_c $ and $ g_s $ stand for the remaining terms,
which due to Young's inequality for convolutions 
satisfy
\begin{eqnarray*}
\| g_c(\widehat{v}_c,\widehat{v}_s) \|_{L^1}& \leq  & C(\|\widehat{v}_c\|_{L^1}^6 + \|\widehat{v}_c\|_{L^1}^4
\|\widehat{v}_s\|_{L^1}
 + \|\widehat{v}_c\|_{L^1}^2 \|\widehat{v}_s\|_{L^1}^2+\|\widehat{v}_s\|_{L^1}^3 ),\\
\| g_s(\widehat{v}_c,\widehat{v}_s)  \|_{L^1} & \leq & C(\|\widehat{v}_c\|_{L^1}^5 + \|\widehat{v}_c\|_{L^1}^3 
\|\widehat{v}_s\|_{L^1}+ \|\widehat{v}_c\|_{L^1} \|\widehat{v}_s\|_{L^1}^2+\|\widehat{v}_s\|_{L^1}^3)
\end{eqnarray*}
for sufficiently small $ \|\widehat{v}_s\|_{L^1} $ and $ \|\widehat{v}_s\|_{L^1} $.
We note that for the Ginzburg-Landau equation \eqref{spain}, the bilinear and trilinear terms are the only nonvanishing terms in these expansions.
The splitting is motivated as follows.
If $ \widehat{v}_c $ decays like $ T^{-1/4} $, then $ \widehat{v}_s $, 
which  is expected to be formally slaved to $ \widehat{v}_c $,
decays at least  like $ T^{-1/2} $.
Then $ g_c $ decays like $ T^{-3/2} $ and is therefore irrelevant w.r.t.~the linear  dynamics of $ \widehat{v}_c $.
Here and in the following the decays are referred to the decay of the $L^\infty$-norm of $v$ or the $L^1$-norm of $\hat{v}$, cf. Section \ref{sec7}.

In order to prove the irrelevance of the other terms w.r.t.~the linear dynamics of $ \widehat{v}_c $, 
except for the marginal  one found in Section \ref{sec3},
we make a change of coordinates which removes in the equation for $ \widehat{v}_c $ all terms containing $ \widehat{v}_s $ except in  $ g_c(\widehat{v}_c,\widehat{v}_s) $.
This change of coordinates  motivates the splitting in the equation for $ \widehat{v}_s $ 
and is defined by solving 
\begin{eqnarray} \label{mice56}
0 & = & \widehat{L}_s  \widehat{v}_s  + B_{2,4}(\widehat{v}_c,\widehat{v}_c) + B_{3,4}(\widehat{v}_c,\widehat{v}_c,\widehat{v}_c) 
+ B_{4,4}(\widehat{v}_c,\widehat{v}_c,\widehat{v}_c,\widehat{v}_c) \\ && + B_{2,5}(\widehat{v}_c,\widehat{v}_s)+ B_{3,5}(\widehat{v}_c,\widehat{v}_c,\widehat{v}_s)
 + B_{2,6}(\widehat{v}_s,\widehat{v}_s)  \nonumber 
\end{eqnarray}
w.r.t.~$ \widehat{v}_s $. For small $\widehat{v}_c $  the implicit function theorem can be  applied in $ L^1 \cap L^{\infty} $ since $ \widehat{L}_s(k)  $ is invertible on the range of $\widehat{P}_s(k)$. Hence there exists a solution $  \widehat{v}_s = \widehat{v}_s^*(\widehat{v}_c) $ 
where $ \widehat{v}_s^*(\widehat{v}_c)  $ is arbitrarily smooth from 
$ L^1 \cap L^{\infty} \to L^1 \cap L^{\infty} $
due to the compact support of $ \widehat{v}_c $
and the polynomial character of \eqref{mice56}.
Hence, we have the following estimate 
\begin{equation} \label{dance2}
\| \widehat{v}_s^*(\widehat{v}_c) \|_{L^1} \leq C \| \widehat{v}_c \|_{L^1}^2
\end{equation}
for $  \| \widehat{v}_c \|_{L^1} $ sufficiently small.
We set
\begin{equation}\label{coo22}
 \widehat{v}_c =  \widehat{w}_c, \qquad
 \widehat{v}_s =  \widehat{v}_s^*(\widehat{w}_c)  + \widehat{w}_s.
\end{equation}
As we will see the new variable $ \widehat{w}_s$ decays like $ T^{-5/4} $.
This decay rate allows us to handle all $ \widehat{w}_s$ terms in the equation 
for $ \widehat{w}_c$ immediately as irrelevant.
As before we introduce 
$ \widehat{W}_c $ by $  \widehat{w}_c (k,t)    =  \widehat{W}_c(k,t) \widehat{\varphi}_{1}(k) $, and 
$ \widehat{W}_s $ for  $ |k| \leq k_0/2 $ by $  \widehat{w}_s (k,t)    =  \widehat{W}_s(k,t) \widehat{\varphi}_{2}(k) $.

Applying the transformation \eqref{coo22} we find from
$$ 
 \partial_T \widehat{v}_s =  \widehat{v}_s^{*'}(\widehat{w}_c)  \partial_T \widehat{w}_c +  \partial_T  \widehat{w}_s
$$ 
that 
\begin{eqnarray*}
 \partial_T  \widehat{w}_s & = &  \widehat{L}_s  \widehat{v}_s + \widehat{P}_s \widehat{N}(\widehat{v})
-  \widehat{v}_s^{*'}(\widehat{w}_c)  \partial_T \widehat{w}_c \\ 
& = &   \widehat{L}_s  \widehat{w}_s + ( \widehat{L}_s   \widehat{v}_s^*(\widehat{w}_c) 
 + \widehat{P}_s \widehat{N}(\widehat{v}))
-  \widehat{v}_s^{*'}(\widehat{w}_c)  \partial_T \widehat{w}_c, 
\end{eqnarray*}
where $ \widehat{v}_s^{*'}(\widehat{w}_c) $ is  the Fr\'echet derivative 
at the point $ \widehat{w}_c $
acting on $\partial_T w_c$. For $  \| \widehat{w}_c \|_{L^1} $ sufficiently small, we have
\begin{equation} \label{dance2b}
\| \widehat{v}_s^{*'}(\widehat{w}_c) \partial_T \widehat{w}_c   \|_{L^1} \leq C \| \widehat{w}_c \|_{L^1}
 \| \partial_T \widehat{w}_c   \|_{L^1}.
\end{equation}
By \eqref{mice56} we remove all terms of lower order in 
$ \widehat{L}_s   \widehat{v}_s^*(\widehat{w}_c) 
 + \widehat{P}_s \widehat{N}(\widehat{v}) $, \ie, we have 
$$ 
 \| \widehat{L}_s   \widehat{v}_s^*(\widehat{w}_c) 
 + \widehat{P}_s \widehat{N}(\widehat{v}) \|_{L^1} \leq C  ( 
\|\widehat{w}_c\|_{L^1}^5 + \|\widehat{w}_c\|_{L^1} \|\widehat{w}_s\|_{L^1}+  \|\widehat{w}_s\|_{L^1}^2)
$$ 
for sufficiently small $  \|\widehat{w}_c\|_{L^1} $ and $ \|\widehat{w}_s\|_{L^1} $,
and so we obtain a system 
\begin{equation}\label{shs1}
\begin{array}{lcl}
\partial_T \widehat{w}_c & = & \widehat{L}_c   \widehat{w}_c 
+ M_2(\widehat{w}_c) 
+ \widetilde{B}_{2}(\widehat{w}_c) 
+ \widetilde{B}_{3}(\widehat{w}_c) \\&& 
+ \widetilde{B}_{4}(\widehat{w}_c) 
+ \widetilde{B}_{5}(\widehat{w}_c) +\widetilde{g}_c(\widehat{w}_c,\widehat{w}_s),\\ 
\partial_T  \widehat{w}_s  & = & \widehat{L}_s \widehat{w}_s + \widetilde{g}_s(\widehat{w}_c,\widehat{w}_s),
\end{array}
\end{equation}
where $ M_2 $ is a bilinear mapping, and where the $ \widetilde{B}_{m} $ are symmetric $m$-linear mappings. The remaining terms in the $  \widehat{w}_c $-equation are collected in 
$  \widetilde{g}_c(\widehat{w}_c,\widehat{w}_s)  $ with 
$$
 \| \widetilde{g}_c(\widehat{w}_c,\widehat{w}_s) \|_{L^1} \leq  C ( \|\widehat{w}_c\|_{L^1}^6 + \|\widehat{w}_c\|_{L^1} \| \widehat{w}_s\|_{L^1}
 +\|\widehat{w}_s\|_{L^1}^2 )
 $$
 for sufficiently small $  \|\widehat{w}_c\|_{L^1} $ and $ \|\widehat{w}_s\|_{L^1} $.
The separation of the quadratic terms in $ M_2 $ and $  \widetilde{B}_{2}(\widehat{w}_c)$
is made to distinguish the marginal term from the irrelevant quadratic ones, \ie, 
 $ M_2 $ will be the counterpart to the marginal term  
$ -\frac{3}{2}
\sqrt{3}  \partial_X \left(( \partial_X \varphi)^2\right) $
in \eqref{bienne}. 
By construction of the transform \eqref{coo22} we have 
\begin{eqnarray*}
\| \widetilde{g}_s(\widehat{w}_c,\widehat{w}_s) \|_{L^1} & \leq  & C(
\|\widehat{w}_c\|_{L^1}^5 + \|\widehat{w}_c\|_{L^1} \|\widehat{w}_s\|_{L^1}+  \|\widehat{w}_s\|_{L^1}^2+\|\widehat{w}_c\|_{L^1}
\| \partial_T \widehat{w}_c\|_{L^1})
\end{eqnarray*}
for sufficiently small $  \|\widehat{w}_c\|_{L^1} $ and $ \|\widehat{w}_s\|_{L^1} $.
The terms in $  \widetilde{g}_s(\widehat{w}_c,\widehat{w}_s) $ all will turn out to be irrelevant w.r.t. the linear dynamics.
The term $ \partial_T \widehat{w}_c $ on the right hand side of the $ \widehat{w}_s $-equation
can be expressed by the right hand side of the $ \widehat{w}_c $-equation, such that \eqref{shs1} is a well-defined initial value problem.
However, we keep the notation with $ \partial_T \widehat{w}_c $
for the subsequent estimates.

The $ m $-linear terms $ \widetilde{B}_{m} $ are of the form
\begin{eqnarray*}
 \widetilde{B}_{2}(\widehat{w}_c)(k)   & = &  \left(  \int K_2(k,k-l,l)  \widehat{W}_c(k-l) \widehat{W}_c(l) \mathrm{d}l \right)\ \widehat{\varphi}_{1}(k) ,
   \\
 \widetilde{B}_{3}(\widehat{w}_c)(k)  & = & \left(  \int K_3(k,k-l,l-l_1,l_1)  \widehat{W}_c(k-l) \widehat{W}_c(l-l_1) \widehat{W}_c(l_1)\mathrm{d}l_1\mathrm{d}l \right)\ \widehat{\varphi}_{1}(k) ,
\end{eqnarray*}
and similarly for $  \widetilde{B}_{4} $ and $  \widetilde{B}_{5} $. 
The marginal term $ M_2 $ corresponding to $ -\frac{3}{2}
\sqrt{3}\partial_X \left(( \partial_X \varphi)^2\right) $ is given by 
\begin{equation} \label{marginalterm}
 {M}_{2}(\widehat{w}_c)(k)    =     \left(  \int K^*(k,k-l,l)  \widehat{W}_c(k-l) \widehat{W}_c(l) \mathrm{d}l \right)\ \widehat{\varphi}_{1}(k).
\end{equation}

In order to prove the irrelevance of 
$  \widetilde{B}_{2}, \ldots,  \widetilde{B}_{5} $ and the marginality of $ M_2 $ 
we need:
 \begin{lemma} \label{lem33}
The kernels $ K^* , K_2,\ldots,K_5 $
satisfy
\begin{equation} 
\begin{array}{rcl}
\label{alcu5}
|K^*(k,k_1,k_2) & \leq & C | k|| k_1 ||k_2|, \\ 
|K_2(k,k_1,k_2) | & \leq & C (|k|^4+|k_1|^4+|k_2|^4), \\
|K_3(k,k_1,k_2,k_3) | & \leq & C (|k|^3+|k_1|^3+|k_2|^3+|k_3|^3), \\
|K_4(k,k_1,k_2,k_3,k_4) | & \leq & C (|k|^2+|k_1|^2+|k_2|^2+|k_3|^2+ |k_4|^2), \\
|K_5(k,k_1,k_2,k_3,k_4,k_5) | & \leq & C (|k|+|k_1|+|k_2|+|k_3|+ |k_4|+|k_5|).
\end{array}
\end{equation}
for $  k, k_1,k_2,k_3,k_4,k_5 \to 0 $.
\end{lemma}
\begin{proof}
The simple argument is that \eqref{vrvi} and \eqref{shs1}  describe the same system with different variables. 
Thus, in both representations we must have in particular  the same asymptotic behavior.
Hence, the estimates \eqref{alcu5} must hold. For those who are not convinced by this argument the necessary calculations 
for obtaining \eqref{alcu5} can be found in Appendix \ref{appB}.
\end{proof}

\section{The nonlinear decay estimates}
\label{sec7}

With the preparations from Section \ref{secneu3}
we proceed as in \cite{MSU01} and 
consider the variation of constants formula 
\begin{equation}\label{shs136}
\begin{array}{lcl}
\widehat{w}_c(T) & = & e^{T \widehat{L}_c}  \widehat{w}_c(0)  + \int_0^T e^{(T-\tau) \widehat{L}_c} (M_2(\widehat{w}_c) 
+ \widetilde{B}_{2}(\widehat{w}_c) \\ && \qquad
+ \widetilde{B}_{3}(\widehat{w}_c) 
+ \widetilde{B}_{4}(\widehat{w}_c) 
+ \widetilde{B}_{5}(\widehat{w}_c) +\widetilde{g}_c(\widehat{w}_c,\widehat{w}_s))(\tau) \mathrm{d} \tau,\\
 \widehat{w}_s(T)  & = & e^{T \widehat{L}_s } \widehat{w}_s(0) + \int_0^T e^{ (T-\tau)\widehat{L}_s }
 \widetilde{g}_s(\widehat{w}_c,\widehat{w}_s)(\tau) \mathrm{d} \tau.
\end{array}
\end{equation}
 for \eqref{shs1}. 
 In the following we use the abbreviations
\begin{eqnarray*}
a_{c,\nu}(T) & = & \sup_{0\leq \tau\leq T} \|(1{+}\tau)^{{\nu}/{4}} |k|^{\nu} \widehat{w}_c(\tau)\|_{L^{\infty}},\\
b_{c,\nu}(T) & = & \sup_{0\leq \tau\leq T} \|(1{+}\tau)^{(\nu+1)/4} |k|^{\nu}  \widehat{w}_c(\tau)\|_{L^{1}}, \\
a_{s}(T) & = & \sup_{0\leq \tau\leq T} \| (1{+}\tau)^{{\nu^*}/{4}} \widehat{w}_s(\tau)\|_{L^{\infty}},\\
b_{s}(T) & = & \sup_{0\leq \tau\leq T} \|(1{+}\tau)^{(\nu^*+1)/4} \widehat{w}_s(\tau)\|_{L^{1}},
\end{eqnarray*} 
with $ \nu \in \{0,1,2,3,\nu^*\} $ where $ \nu^* $ is a fixed real number with $ \nu^* < 4 $ which can and will be  chosen arbitrarily close to $ 4 $. 
Moreover, many different constants are denoted with the same symbol $ C $, 
if they can be chosen independently of $ a_{c,0}(T), \ldots,b_s(T) $, and $ T $.
The $ a_{c,0}(T), \ldots,b_s(T) $ will be small and so we assume that they are all 
 smaller than one.

It is sufficient to control $ a_{c,0}(T) $, $ b_{c,0}(T) $, $ a_{c,\nu^*}(T) $, $ b_{c,\nu^*}(T) $, $ a_{s}(T)$, and $ b_{s}(T)$.
We have for instance
\begin{eqnarray*}
a_{c,1}(T) & = & \sup_{0\leq \tau\leq T} \|(1{+}\tau)^{1/{4}} |k| \widehat{w}_c(\tau)\|_{L^{\infty}}
\\ & \leq & \sup_{0\leq \tau\leq T} \|(1{+}\tau)^{1/{4}} (|k|^{\nu^*} |\widehat{w}_c(\tau)|)^{{1}/{\nu^*}}
|\widehat{w}_c(\tau)|^{1- {1}/{\nu^*}} \|_{L^{\infty}} \\ 
& \leq & \sup_{0\leq \tau\leq T} (1{+}\tau)^{1/{4}} ((1{+}\tau)^{-\nu^*/{4}} )^{{1}/{\nu^*}}
a_{c,\nu^*}^{{1}/{\nu^*}}(T) a_{c,0}^{1- {1}/{\nu^*}}(T) \\
& \leq &  C a_{c,\nu^*}^{{1}/{\nu^*}}(T) a_{c,0}^{1- {1}/{\nu^*}}(T)\leq   C (a_{c,\nu^*}(T)+ a_{c,0}(T)).
\end{eqnarray*} 
From \eqref{alcu5} and Young's inequality for convolutions
we find
\begin{eqnarray*}
\|\widetilde{B}_2(\widehat{w}_c) \|_{L^{1}} & \leq & C \| |k|^4 \widehat{w}_c\|_{L^{1}}\| \widehat{w}_c\|_{L^{1}},\\
\|\widetilde{B}_3(\widehat{w}_c) \|_{L^{1}} & \leq & C \| |k|^3 \widehat{w}_c\|_{L^{1}}\| \widehat{w}_c\|_{L^{1}}^2,\\
\|\widetilde{B}_4(\widehat{w}_c) \|_{L^{1}} & \leq & C \| |k|^2 \widehat{w}_c\|_{L^{1}}\| \widehat{w}_c\|_{L^{1}}^3,\\
\|\widetilde{B}_5(\widehat{w}_c) \|_{L^{1}} & \leq & C \| |k| \widehat{w}_c\|_{L^{1}}\| \widehat{w}_c\|_{L^{1}}^4,
\end{eqnarray*} 
and recall 
\begin{eqnarray*}
\|\widetilde{g}_c(\widehat{w}_c,\widehat{w}_s)\|_{L^{1}} & \leq & C (\| \widehat{w}_c\|_{L^{1}}^6 + \| \widehat{w}_c\|_{L^{1}} \| \widehat{w}_s\|_{L^{1}}   + \| \widehat{w}_s\|_{L^{1}}^2) ,\\
\|\widetilde{g}_s(\widehat{w}_c,\widehat{w}_s)\|_{L^{1}} & \leq & C (\| \widehat{w}_c\|_{L^{1}}^5 + \| \widehat{w}_c\|_{L^{1}} \| \widehat{w}_s\|_{L^{1}}+ \| \widehat{w}_s\|_{L^{1}}^2 +\|\widehat{w}_c\|_{L^{1}}\| \partial_T \widehat{w}_c\|_{L^{1}} ),
\end{eqnarray*} 
Due to the convolution structure of all terms occurring in our calculations 
we have, again by  \eqref{alcu5} and Young's inequality for convolutions, that 
\begin{eqnarray*}
\|\widetilde{B}_2(\widehat{w}_c) \|_{L^{\infty}} & \leq & C \| |k|^4 \widehat{w}_c\|_{L^{\infty}}\| \widehat{w}_c\|_{L^{1}}
%+ \| |k|^4 \widehat{w}_c\|_{L^{1}}\| \widehat{w}_c\|_{L^{\infty}})
,\\
\|\widetilde{B}_3(\widehat{w}_c) \|_{L^{\infty}} & \leq & C \| |k|^3 \widehat{w}_c\|_{L^{\infty}}\| \widehat{w}_c\|_{L^{1}}^2 
%+ \| |k|^3 \widehat{w}_c\|_{L^{1}}\| \widehat{w}_c\|_{L^{1}}\| \widehat{w}_c\|_{L^{\infty}})
,\\
\|\widetilde{B}_4(\widehat{w}_c) \|_{L^{\infty}} & \leq & C \| |k|^2 \widehat{w}_c\|_{L^{\infty}}\| \widehat{w}_c\|_{L^{1}}^3
%+ \| |k|^2 \widehat{w}_c\|_{L^{1}}\| \widehat{w}_c\|_{L^{1}}^2 \| \widehat{w}_c\|_{L^{\infty}})
,\\
\|\widetilde{B}_5(\widehat{w}_c) \|_{L^{\infty}} & \leq & C \| |k| \widehat{w}_c\|_{L^{\infty}}\| \widehat{w}_c\|_{L^{1}}^4
%+ \| |k| \widehat{w}_c\|_{L^{1}}\| \widehat{w}_c\|_{L^{1}}^3 \| \widehat{w}_c\|_{L^{\infty}})
,\\
\|\widetilde{g}_c(\widehat{w}_c,\widehat{w}_s)\|_{L^{\infty}} & \leq & C (\| \widehat{w}_c\|_{L^{1}}^5 \| \widehat{w}_c\|_{L^{\infty}}
%+ \| \widehat{w}_c\|_{L^{1}} \| \widehat{w}_s\|_{L^{\infty}}  
 + \| \widehat{w}_c\|_{L^{\infty}} \| \widehat{w}_s\|_{L^{1}}   + \| \widehat{w}_s\|_{L^{1}}\| \widehat{w}_s\|_{L^{\infty}}) ,\\
\|\widetilde{g}_s(\widehat{w}_c,\widehat{w}_s)\|_{L^{\infty}} & \leq & C (\| \widehat{w}_c\|_{L^{1}}^4 \| \widehat{w}_c\|_{L^{\infty}}
%+ \| \widehat{w}_c\|_{L^{1}} \| \widehat{w}_s\|_{L^{\infty}}  \\ && \qquad 
+ \| \widehat{w}_c\|_{L^{\infty}} \| \widehat{w}_s\|_{L^{1}}  \\ && \qquad + \| \widehat{w}_s\|_{L^{1}}\| \widehat{w}_s\|_{L^{\infty}}
% \\ && \qquad   \qquad +\|\widehat{w}_c\|_{L^{\infty}}\| \partial_T \widehat{w}_c\|_{L^{1}} 
+\|\widehat{w}_c\|_{L^{1}}\| \partial_T \widehat{w}_c\|_{L^{\infty}} 
).
\end{eqnarray*}

\subsection{The diffusive modes}

Since $ \widehat{w}_c $ has compact support in Fourier space 
$ k^4 \widehat{w}_c $ can be estimated in terms of $ |k|^{\nu} \widehat{w}_c $
for every $ \nu \in [0,4) $, in particular for $ \nu = \nu^* $. For $ \nu \in (3,4) $ we have:

\noindent
{\bf a)} We estimate 
\begin{eqnarray*}
&& 
\biggl\|
\int^T_0 e^{(T-\tau) \widehat{L}_c}
\widetilde{B}_2(\widehat{w}_c)(\tau) \mathrm{d} \tau 
\biggr\|_{L^{\infty}}\\
&  \leq &  \int^{T}_0 \|e^{(T-\tau) \widehat{L}_c}
\|_{L^{\infty}\to L^{\infty}}   \| \widetilde{B}_2(\widehat{w}_c)(\tau)  \|_{L^{\infty}} 
  \mathrm{d} \tau\\
& \leq &  C  \int^T_0 
       (1{+}\tau)^{-(\nu+1)/{4}} \mathrm{d} \tau \cdot   (a_{c,\nu}(T) b_{c,0}(T))
       %+ b_{c,\nu}(T) a_{c,0}(T))
       \\& \leq &C a_{c,\nu}(T) b_{c,0}(T).%+ b_{c,\nu}(T) a_{c,0}(T)) .
\end{eqnarray*}
Similarly, we find 
\begin{eqnarray*}
\biggl\|
\int^T_0 e^{(T-\tau) \widehat{L}_c}\widetilde{B}_3(\widehat{w}_c)(\tau) 
 \mathrm{d} \tau 
\biggr\|_{L^{\infty}} & \leq & 
C a_{c,3}(T) b_{c,0}^2(T)%+ b_{c,3}(T) a_{c,0}(T)b_{c,0}(T))
, \\
\biggl\|\int^T_0 e^{(T-\tau) \widehat{L}_c}\widetilde{B}_4(\widehat{w}_c)(\tau) 
 \mathrm{d} \tau 
\biggr\|_{L^{\infty}} & \leq & 
C a_{c,2}(T) b_{c,0}^3(T)%+ b_{c,2}(T) a_{c,0}(T)b_{c,0}^2(T))
, \\
\biggl\|\int^T_0 e^{(T-\tau) \widehat{L}_c}\widetilde{B}_5(\widehat{w}_c)(\tau) 
 \mathrm{d} \tau 
\biggr\|_{L^{\infty}} & \leq & 
Ca_{c,1}(T) b_{c,0}^4(T)%+ b_{c,1}(T) a_{c,0}(T)b_{c,0}^3(T))
,
\end{eqnarray*}
and
\begin{eqnarray*}
\biggl\|
\int^T_0 e^{(T-\tau) \widehat{L}_c}\widetilde{g}_c(\widehat{w}_c,\widehat{w}_s)
(\tau) 
 \mathrm{d} \tau 
\biggr\|_{L^{\infty}}
& \leq & C (b_{c,0}^5(T) a_{c,0}(T)%+ b_{c,0}(T) a_{s}(T)\\ && 
+ a_{c,0}(T) b_{s}(T)
+a_{s}(T) b_{s}(T)).
\end{eqnarray*}

\noindent
{\bf b)}  Next we estimate 
\begin{eqnarray*}
&& (1{+}T)^{1/4}
\biggl\| \int^T_0 e^{(T-\tau) \widehat{L}_c}
\widetilde{B}_2(\widehat{w}_c)(\tau) \mathrm{d} \tau   \biggr\|_{L^{1}}\\
&  \leq &  (1{+}T)^{1/4} \int^{T}_0 \|e^{(T-\tau) \widehat{L}_c}
\|_{L^{\infty}\to L^{1}}   \| \widetilde{B}_2(\widehat{w}_c)(\tau)  \|_{L^{\infty}} 
  \mathrm{d} \tau\\
& \leq &   C (1{+}T)^{1/4} \int_0^{T} (T-\tau)^{-1/4} (1{+}\tau)^{-(\nu + 1)/{4}}
        \mathrm{d} \tau \cdot    (a_{c,\nu}(T) b_{c,0}(T))%+ b_{c,\nu}(T) a_{c,0}(T))
     \\
&  \leq  &C  (1{+}T)^{1/4} \int^{T/2}_0(T/2)^{-1/4}  (1{+}\tau)^{-(\nu+1)/{4}} \mathrm{d} \tau  \cdot    (a_{c,\nu}(T) b_{c,0}(T))%+ b_{c,\nu}(T) a_{c,0}(T))
\\ &&
+C (1{+}T)^{1/4} \int_{T/2}^{T}(T-\tau)^{-1/4}  (1{+}T/2)^{-(\nu+1)/{4}} \mathrm{d} \tau \cdot    (a_{c,\nu}(T) b_{c,0}(T))%+ b_{c,\nu}(T) a_{c,0}(T))
\\&\leq &C    a_{c,\nu}(T) b_{c,0}(T)%+ b_{c,\nu}(T) a_{c,0}(T)).
\end{eqnarray*} 
It is easily verified that the same technique of splitting the integral $\int_0^T =\int_0^{T/2}+\int_{T/2}^T$ can be used to show that
\begin{equation}\label{splitting_trick}
\begin{split}
 &\text{for all }\alpha,\gamma \geq 0,\beta\in (0,1)\text{ with }\alpha-\beta-\gamma\leq -1\text{ there exists }C>0\\
 &\text{such that for all }T>0 \text{ we have }(1+T)^\alpha\int_0^T(T-\tau)^{-\beta}(1+\tau)^{-\gamma }\,\mathrm{d}\tau \leq C.
\end{split}
\end{equation}
Similarly, we find 
\begin{eqnarray*}
 (1{+}T)^{1/4}\biggl\|
\int^T_0 e^{(T-\tau) \widehat{L}_c}\widetilde{B}_3(\widehat{w}_c)(\tau) 
 \mathrm{d} \tau 
\biggr\|_{L^{1}} & \leq & 
Ca_{c,3}(T) b_{c,0}^2(T)%+ b_{c,3}(T) a_{c,0}(T)b_{c,0}(T))
, \\
 (1{+}T)^{1/4}\biggl\|\int^T_0 e^{(T-\tau) \widehat{L}_c}\widetilde{B}_4(\widehat{w}_c)(\tau) 
 \mathrm{d} \tau 
\biggr\|_{L^{1}} & \leq & 
Ca_{c,2}(T) b_{c,0}^3(T)%+ b_{c,2}(T) a_{c,0}(T)b_{c,0}^2(T))
, \\
 (1{+}T)^{1/4}\biggl\|\int^T_0 e^{(T-\tau) \widehat{L}_c}\widetilde{B}_5(\widehat{w}_c)(\tau) 
 \mathrm{d} \tau 
\biggr\|_{L^{1}} & \leq & 
Ca_{c,1}(T) b_{c,0}^4(T)%+ b_{c,1}(T) a_{c,0}(T)b_{c,0}^3(T))
,
\end{eqnarray*}
and
\begin{eqnarray*}
(1{+}T)^{1/4} \biggl\|
 \int^T_0 e^{(T-\tau) \widehat{L}_c}\widetilde{g}_c(\widehat{w}_c,\widehat{w}_s)
(\tau) 
 \mathrm{d} \tau 
\biggr\|_{L^{1}}
& \leq & C (b_{c,0}^5(T) a_{c,0}(T)%+ b_{c,0}(T) a_{s}(T)
 + a_{c,0}(T) b_{s}(T) \\ && \qquad
+a_{s}(T) b_{s}(T)).
\end{eqnarray*}

\noindent
{\bf c)} We estimate
\begin{eqnarray*}
&& 
(1{+}T)^{\nu/4} \biggl\|
\int^T_0 e^{(T-\tau) \widehat{L}_c}
|k|^{\nu}  \widetilde{B}_2(\widehat{w}_c)(\tau) \mathrm{d} \tau 
\biggr\|_{L^{\infty}}\\
&  \leq & (1{+}T)^{\nu/4}  \int^{T}_0 \|e^{(T-\tau) \widehat{L}_c}|k|^{\nu}
\|_{L^{\infty}\to L^{\infty}}   \| \widetilde{B}_2(\widehat{w}_c)(\tau)   \|_{L^{\infty}} 
  \mathrm{d} \tau\\
& \leq &  (1{+}T)^{\nu/4}  C \int^T_0  (T-\tau)^{-\nu/4}(1{+}\tau)^{-(\nu+1)/{4}}
        \mathrm{d} \tau \cdot (a_{c,\nu}(T) b_{c,0}(T))%+ b_{c,\nu}(T) a_{c,0}(T))
              \\&\leq& C a_{c,\nu}(T) b_{c,0}(T)%+ b_{c,\nu}(T) a_{c,0}(T))
\end{eqnarray*}
using \eqref{splitting_trick}.
Similarly, we find 
\begin{eqnarray*}
 (1{+}T)^{\nu/4}\biggl\|
\int^T_0 e^{(T-\tau) \widehat{L}_c}|k|^{\nu} \widetilde{B}_3(\widehat{w}_c)(\tau) 
 \mathrm{d} \tau 
\biggr\|_{L^{\infty}} & \leq & 
Ca_{c,3}(T) b_{c,0}^2(T)%+ b_{c,3}(T) a_{c,0}(T)b_{c,0}(T))
, \\
 (1{+}T)^{\nu/4}\biggl\|\int^T_0 e^{(T-\tau) \widehat{L}_c}|k|^{\nu} \widetilde{B}_4(\widehat{w}_c)(\tau) 
 \mathrm{d} \tau 
\biggr\|_{L^{\infty}} & \leq & 
Ca_{c,2}(T) b_{c,0}^3(T)%+ b_{c,2}(T) a_{c,0}(T)b_{c,0}^2(T))
, \\
 (1{+}T)^{\nu/4}\biggl\|\int^T_0 e^{(T-\tau) \widehat{L}_c}|k|^{\nu} \widetilde{B}_5(\widehat{w}_c)(\tau) 
 \mathrm{d} \tau 
\biggr\|_{L^{\infty}} & \leq & 
Ca_{c,1}(T) b_{c,0}^4(T)%+ b_{c,1}(T) a_{c,0}(T)b_{c,0}^3(T))
,
\end{eqnarray*}
and
\begin{eqnarray*}
(1{+}T)^{\nu/4} \biggl\|
 \int^T_0 e^{(T-\tau) \widehat{L}_c}|k|^{\nu} \widetilde{g}_c(\widehat{w}_c,\widehat{w}_s)
(\tau) 
 \mathrm{d} \tau 
\biggr\|_{L^{\infty}}
& \leq & C (b_{c,0}^5(T) a_{c,0}(T)%+ b_{c,0}(T) a_{s}(T)
+ a_{c,0}(T) b_{s}(T)\\ &&  \qquad
+a_{s}(T) b_{s}(T)).
\end{eqnarray*}

\noindent
{\bf 
d)} The last estimate  for the diffusive part is
\begin{eqnarray*}
&& 
(1{+}T)^{(\nu+1)/4} 
\biggl\| \int^T_0 e^{(T-\tau) \widehat{L}_c}|k|^{\nu}
\widetilde{B}_2(\widehat{w}_c)(\tau) \mathrm{d} \tau   \biggr\|_{L^{1}}\\
&  \leq &  (1{+}T)^{(\nu+1)/4} \int^{T-1}_0 \|e^{(T-\tau) \widehat{L}_c}|k|^{\nu}
\|_{L^{\infty}\to L^{1}}   \| \widetilde{B}_2(\widehat{w}_c)(\tau)  \|_{L^{\infty}} 
  \mathrm{d} \tau\\
  && +  (1{+}T)^{(\nu+1)/4} \int_{T-1}^T \|e^{(T-\tau) \widehat{L}_c}|k|^{\nu}
\|_{L^{1}\to L^{1}}   \| \widetilde{B}_2(\widehat{w}_c)(\tau)  \|_{L^{1}} 
  \mathrm{d} \tau\\
& \leq &  (1{+}T)^{(\nu+1)/4} C\int^{T-1}_0(T{-}\tau)^{-(\nu+1)/4}  (1{+}\tau)^{-(\nu+1)/{4}}
        \mathrm{d} \tau \cdot (a_{c,\nu}(T) b_{c,0}(T))%+ b_{c,\nu}(T) a_{c,0}(T)) 
        \\
       &&+(1{+}T)^{(\nu+1)/4} C\int^T_{T-1}(T{-}\tau)^{-\nu/4}  (1{+}\tau)^{-(\nu+2)/4}
     \mathrm{d} \tau \cdot (b_{c,\nu}(T) b_{c,0}(T) ) 
       \\ & \leq & s_1 + C b_{c,\nu}(T) b_{c,0}(T) .
\end{eqnarray*}
We split  $ \int^{T-1}_0 \ldots= \int^{T/2}_0 \ldots+\int^{T-1}_{T/2} \ldots $, resp. $ s_1 
= s_2 + s_3$, and find 
\begin{eqnarray*}
s_2 & \leq  &  (1{+}T)^{(\nu+1)/4} \int^{T/2}_0(T/2)^{-(\nu+1)/4}  (1{+}\tau)^{-(\nu+1)/{4}} \mathrm{d} \tau  \cdot  (a_{c,\nu}(T) b_{c,0}(T)).%+ b_{c,\nu}(T) a_{c,0}(T)).
\end{eqnarray*} 
Moreover, 
\begin{eqnarray*}
s_3 & \leq  &  (1{+}T)^{(\nu+1)/4} \int_{T/2}^{T-1}(T-\tau)^{-(\nu+1)/4}  (1{+}T/2)^{-(\nu+1)/{4}}  \mathrm{d} \tau  \cdot (a_{c,\nu}(T) b_{c,0}(T))%+ b_{c,\nu}(T) a_{c,0}(T)) 
\end{eqnarray*} 
such that finally 
$$ s_1 \leq C a_{c,\nu}(T) b_{c,0}(T)%+ b_{c,\nu}(T) a_{c,0}(T)) 
.$$

Similarly, we find 
\begin{eqnarray*}
 (1{+}T)^{(\nu+1)/4}\biggl\|
\int^T_0 e^{(T-\tau) \widehat{L}_c}|k|^{\nu} \widetilde{B}_3(\widehat{w}_c)(\tau) 
 \mathrm{d} \tau 
\biggr\|_{L^{1}} & \leq & 
C(a_{c,3}(T) b_{c,0}^2(T)%+ b_{c,3}(T) a_{c,0}(T)b_{c,0}(T) \\ && \qquad 
+ b_{c,3}(T) b_{c,0}^2(T)), \\
 (1{+}T)^{(\nu+1)/4}\biggl\|\int^T_0 e^{(T-\tau) \widehat{L}_c}|k|^{\nu} \widetilde{B}_4(\widehat{w}_c)(\tau) 
 \mathrm{d} \tau 
\biggr\|_{L^{1}} & \leq & 
C(a_{c,2}(T) b_{c,0}^3(T)%+ b_{c,2}(T) a_{c,0}(T)b_{c,0}^2(T)\\ && \qquad 
+ b_{c,2}(T) b_{c,0}^3(T)), \\
 (1{+}T)^{(\nu+1)/4}\biggl\|\int^T_0 e^{(T-\tau) \widehat{L}_c}|k|^{\nu} \widetilde{B}_5(\widehat{w}_c)(\tau) 
 \mathrm{d} \tau 
\biggr\|_{L^{1}} & \leq & 
C(a_{c,1}(T) b_{c,0}^4(T)%+ b_{c,1}(T) a_{c,0}(T)b_{c,0}^3(T)\\ && \qquad 
+ b_{c,1}(T) b_{c,0}^4(T)),
\end{eqnarray*}
and
\begin{eqnarray*}
(1{+}T)^{(\nu+1)/4} \biggl\|
 \int^T_0 e^{(T-\tau) \widehat{L}_c}|k|^{\nu} \widetilde{g}_c(\widehat{w}_c,\widehat{w}_s)
(\tau) 
 \mathrm{d} \tau 
\biggr\|_{L^{1}}
& \leq & C (b_{c,0}^5(T) a_{c,0}(T)%+ b_{c,0}(T) a_{s}(T)
\\ && + a_{c,0}(T) b_{s}(T)
+a_{s}(T) b_{s}(T) \\ && + b_{c,0}^6(T) + b_{c,0}(T)b_{s}(T)  +b_{s}^2(T)) .
\end{eqnarray*}

 \subsection{Handling of the marginal terms}
Now we come to the handling of the marginally stable term $ {M}_2(\widehat{w}_c) $ defined in \eqref{marginalterm}. 

\noindent
{\bf a)} 
We find
\begin{eqnarray*}
&& 
\biggl\|
\int^T_0 e^{(T-\tau) \widehat{L}_c}
{M}_2(\widehat{w}_c)(\tau) \mathrm{d} \tau 
\biggr\|_{L^{\infty}}\\
&  \leq & C \int^{T}_0 \|e^{(T-\tau) \widehat{L}_c} |k|
\|_{L^{\infty}\to L^{\infty}}   \| k \widehat{w}_c(\tau)  \|_{L^{\infty}} \| k  \widehat{w}_c(\tau)  \|_{L^1}
  \mathrm{d} \tau\\
& \leq &  C  \int^T_0  (T-\tau)^{-1/4}
       (1{+}\tau)^{-3/4} \mathrm{d} \tau \cdot   a_{c,1}(T) b_{c,1}(T)
       \\& \leq &C  a_{c,1}(T) b_{c,1}(T),
\end{eqnarray*}
where we used \eqref{splitting_trick}.

\noindent
{\bf b)} Next we have 
\begin{eqnarray*}
&& 
\biggl\|
(1+T)^{1/4}\int^T_0 e^{(T-\tau) \widehat{L}_c}
{M}_2(\widehat{w}_c)(\tau) \mathrm{d} \tau 
\biggr\|_{L^{1}}\\
&  \leq & C (1+T)^{1/4} \int^{T}_0 \|e^{(T-\tau) \widehat{L}_c} |k|
\|_{L^{\infty}\to L^{1}}   \| k \widehat{w}_c(\tau)  \|_{L^{\infty}} \| k  \widehat{w}_c(\tau)  \|_{L^1}
  \mathrm{d} \tau\\
& \leq &  C (1+T)^{1/4}  \int^T_0  (T-\tau)^{-1/2}
       (1{+}\tau)^{-3/4} \mathrm{d} \tau \cdot   a_{c,1}(T) b_{c,1}(T)
%       \\
%& \leq &  C (1+T)^{1/4} ( \int^{T/2}_0  (T/2)^{-1/2}
%       (1{+}\tau)^{-3/4} \mathrm{d} \tau \\ && \qquad \qquad \qquad +  \int_{T/2}^T  (T-\tau)^{-1/2}
%       (1{+}T/2)^{-3/4} \mathrm{d} \tau) \cdot   a_{c,1}(T) b_{c,1}(T)
       \\& \leq &C  a_{c,1}(T) b_{c,1}(T),
\end{eqnarray*}
again using \eqref{splitting_trick}.

\noindent
{\bf c)} Moreover, with a $ \theta \in (0, 4 -\nu^*) $ we estimate
\begin{eqnarray*}
&& 
\biggl\|
(1+T)^{\nu^*/4}\int^T_0 e^{(T-\tau) \widehat{L}_c} |k|^{\nu^*}
{M}_2(\widehat{w}_c)(\tau) \mathrm{d} \tau 
\biggr\|_{L^{\infty}}\\
&  \leq & (1+T)^{\nu^*/4} \int^{T}_0 \|e^{(T-\tau) \widehat{L}_c} |k|^{\nu^*+ \theta}
\|_{L^{\infty}\to L^{\infty}}   \| |k|^{2-\theta} \widehat{w}_c(\tau)  \|_{L^{\infty}} \| k  \widehat{w}_c(\tau)  \|_{L^1}
  \mathrm{d} \tau\\
& \leq &  C (1+T)^{\nu^*/4}  \int^T_0  (T-\tau)^{-(\nu^*+ \theta)/4}
       (1{+}\tau)^{-(1-\theta/4)} \mathrm{d} \tau \cdot   a_{c,2-\theta}(T) b_{c,1}(T)
%       \\
%& \leq &  C (1+T)^{\nu^*/4} ( \int^{T/2}_0  (T/2)^{-(\nu^*+ \theta)/4}
%       (1{+}\tau)^{-(1-\theta/4)} \mathrm{d} \tau \\ && \qquad \qquad \qquad +  \int_{T/2}^T  (T-\tau)^{-(\nu^*+ \theta)/4}
%       (1{+}T/2)^{-(1-\theta/4)} \mathrm{d} \tau) \cdot   a_{c,2-\theta}(T) b_{c,1}(T)
       \\& \leq &C  a_{c,2-\theta}(T) b_{c,1}(T),
\end{eqnarray*}
again using \eqref{splitting_trick}.

\noindent
{\bf d)} For the marginally stable term finally again  with a $ \theta \in (0, 4 -\nu^*) $ we estimate
\begin{eqnarray*}
&& 
\biggl\|
(1+T)^{\nu^*/4}\int^T_0 e^{(T-\tau) \widehat{L}_c} |k|^{\nu^*}
{M}_2(\widehat{w}_c)(\tau) \mathrm{d} \tau 
\biggr\|_{L^{1}}\\
&  \leq & (1+T)^{(\nu^*+1)/4} \int^{T/2}_0 \|e^{(T-\tau) \widehat{L}_c} |k|^{\nu^*+1+ \theta}
\|_{L^{\infty}\to L^{1}}   \| |k|^{1- \theta} \widehat{w}_c(\tau)  \|_{L^{\infty}} \| k  \widehat{w}_c(\tau)  \|_{L^1}
  \mathrm{d} \tau
 \\ 
&&  + (1+T)^{(\nu^*+1)/4} \int^{T}_{T/2} \|e^{(T-\tau) \widehat{L}_c} |k|^{\nu^*}
\|_{L^{1}\to L^{1}}   \| |k|^{2} \widehat{w}_c(\tau)  \|_{L^{1}} \| k  \widehat{w}_c(\tau)  \|_{L^1}
  \mathrm{d} \tau
 \\
& \leq &  C (1+T)^{(\nu^*+1)/4}  \int^{T/2}_0  (T/2)^{-(\nu^*+1+\theta)/4}
       (1{+}\tau)^{-(1-\theta)/4} \mathrm{d} \tau \cdot   a_{c,2-\theta}(T) b_{c,1}(T)
       \\ 
     && +   C (1+T)^{(\nu^*+1)/4}  \int^T_{T/2}  (T-\tau)^{-\nu^*/4}
       (1{+}T/2)^{-5/4} \mathrm{d} \tau \cdot   b_{c,2}(T) b_{c,1}(T)       
       \\& \leq &C  (a_{c,2-\theta}(T) b_{c,1}(T)+ b_{c,2}(T) b_{c,1}(T) ).
\end{eqnarray*}

\subsection{The linearly exponentially damped modes}

In the estimates of $
\|\widetilde{g}_s(\widehat{w}_c,\widehat{w}_s)\|_{L^{1}} $ 
and 
$ \|\widetilde{g}_s(\widehat{w}_c,\widehat{w}_s)\|_{L^{\infty}} $  
the new terms  
$$ 
\|\widehat{w}_c\|_{L^{1}}\| \partial_T \widehat{w}_c\|_{L^{1}} \qquad 
\textrm{and} \qquad
%\|\widehat{w}_c\|_{L^{\infty}}\| \partial_T \widehat{w}_c\|_{L^{1}} +
\|\widehat{w}_c\|_{L^{1}}\| \partial_T \widehat{w}_c\|_{L^{\infty}} 
$$
occur. They will be estimated as 
\begin{eqnarray*}
\| \partial_T \widehat{w}_c\|_{L^{1}} & \leq & 
\|\widehat{L}_c \widehat{w}_c \|_{L^{1}} +
\|M_2(\widehat{w}_c) \|_{L^{1}} +
\|\widetilde{B}_2(\widehat{w}_c) \|_{L^{1}} +
\|\widetilde{B}_3(\widehat{w}_c) \|_{L^{1}} \\ && +\|\widetilde{B}_4(\widehat{w}_c) \|_{L^{1}} 
+ \|\widetilde{B}_5(\widehat{w}_c) \|_{L^{1}} + \|\widetilde{g}_c(\widehat{w}_c,\widehat{w}_s)\|_{L^{1}} 
\end{eqnarray*} 
and 
\begin{eqnarray*}
\| \partial_T \widehat{w}_c\|_{L^{\infty}} & \leq &
\|\widehat{L}_c \widehat{w}_c \|_{L^{\infty}} +
\|M_2(\widehat{w}_c) \|_{L^{\infty}} +
\|\widetilde{B}_2(\widehat{w}_c) \|_{L^{\infty}} +
\|\widetilde{B}_3(\widehat{w}_c) \|_{L^{\infty}} \\ && +\|\widetilde{B}_4(\widehat{w}_c) \|_{L^{\infty}} 
+ \|\widetilde{B}_5(\widehat{w}_c) \|_{L^{\infty}} + \|\widetilde{g}_c(\widehat{w}_c,\widehat{w}_s)\|_{L^{\infty}} .
\end{eqnarray*} 
For the right hand side we use the estimates from above and 
\begin{eqnarray*}
\|\widehat{L}_c \widehat{w}_c \|_{L^{1}} & \leq &  C \| |k|^4 \widehat{w}_c\|_{L^{1}} ,\\
\|M_2(\widehat{w}_c) \|_{L^{1}} & \leq & C \| |k|^2 \widehat{w}_c\|_{L^{1}}\| |k|\widehat{w}_c\|_{L^{1}},
\end{eqnarray*} 
and the  similar estimates for $ \|\widehat{L}_c \widehat{w}_c \|_{L^{\infty}} $ and 
$ \|M_2(\widehat{w}_c) \|_{L^{\infty}} $. 
In the subsequent estimates these terms will be collected in $ H_2(T) $ and $ H_4(T) $.

\noindent
{\bf 
a)} Therefore, for the linearly exponentially damped part we  first find
\begin{eqnarray*}
&& 
(1{+}T)^{\nu^*/4} \biggl\|
\int^T_0 e^{(T-\tau)\widehat{L}_s }
\widetilde{g}_s(\widehat{w}_c,\widehat{w}_s)(\tau) \mathrm{d} \tau
\biggr\|_{L^{\infty}}\\
&  \leq & (1{+}T)^{\nu^*/4}  \int^{T}_0 \|e^{(T-\tau)\widehat{L}_s  }
\|_{L^{\infty}\to L^{\infty}}   \|  \widetilde{g}_s(\widehat{w}_c,\widehat{w}_s)(\tau) \|_{L^{\infty}} 
  \mathrm{d} \tau\\
& \leq &  (1{+}T)^{\nu^*/4}  \int^T_0 e^{-\sigma_s(T-\tau)} (1{+}\tau)^{-1} \mathrm{d} \tau \cdot (H_1(T)+H_2(T))\\
\\
&  \leq & 
 C (H_1(T)+H_2(T))
\end{eqnarray*}
due to the uniform boundedness of 
\begin{eqnarray*}
(1{+}T)^{\nu^*/4}  \int^T_0 e^{-\sigma_s(T-\tau)} (1{+}\tau)^{-1} \mathrm{d} \tau &\leq& 
(1{+}T)^{\nu^*/4}  \int^{T/2}_0 e^{-\sigma_sT/2} (1{+}\tau)^{-1} \mathrm{d} \tau \\&&+ (1{+}T)^{\nu^*/4}  
\int_{T/2}^T e^{-\sigma_s(T-\tau)} (1{+}T/2)^{-1} \mathrm{d} \tau 
\end{eqnarray*}
and 
where
\begin{eqnarray*}
H_1(T) & \leq &C( b_{c,0}^4(T) a_{c,0}(T)
%+ b_{c,0}(T) a_{s}(T) 
+ a_{c,0}(T) b_{s}(T)
+a_{s}(T) b_{s}(T)) ,\\ 
H_2(T) & \leq &  C ( a_{c,\nu^*}(T) b_{c,0}(T) + a_{c,2}(T) b_{c,1}(T) b_{c,0}(T)
+ a_{c,3}(T) b_{c,0}(T)^3 \\ &&+ a_{c,2}(T) b_{c,0}(T)^4 + a_{c,1}(T) b_{c,0}(T)^5 + a_{c,0}(T) b_{c,0}(T)^6 \\ &&+ a_{c,0}(T)  b_s(T) b_{c,0}(T)+ a_s(T) b_s(T)b_{c,0}(T)).
\end{eqnarray*}

\noindent
{\bf b)} Secondly, we estimate
\begin{eqnarray*}
&& 
(1{+}T)^{(\nu^*+1)/4} \biggl\|
\int^T_0 e^{(T-\tau)\widehat{L}_s  }
\widetilde{g}_s(\widehat{w}_c,\widehat{w}_s)(\tau)\mathrm{d} \tau
\biggr\|_{L^{1}}\\
&  \leq & (1{+}T)^{(\nu^*+1)/4}  \int^{T}_0 \|e^{ (T-\tau) \widehat{L}_s }
\|_{L^{1}\to L^{1}}   \| \widetilde{g}_s(\widehat{w}_c,\widehat{w}_s) (\tau) \|_{L^{1}} 
  \mathrm{d} \tau\\
& \leq &  (1{+}T)^{(\nu^*+1)/4}  \int^T_0 e^{-\sigma_s(T-\tau)} (1{+}\tau)^{-5/4} \mathrm{d} \tau 
\\&& \qquad \times 
C(H_3(T)+H_4(T))
 \\
&  \leq & C(H_3(T)+H_4(T))
\end{eqnarray*}
due to the uniform boundedness of 
\begin{eqnarray*}
(1{+}T)^{(\nu^*+1)/4}  \int^T_0 e^{-\sigma_s(T-\tau)} (1{+}\tau)^{-5/4} \mathrm{d} \tau &\leq& 
(1{+}T)^{(\nu^*+1)/4}  \int^{T/2}_0 e^{-\sigma_s T/2}(1{+}\tau)^{-5/4} \mathrm{d} \tau \\&&+ (1{+}T)^{(\nu^*+1)/4}  \int_{T/2}^T e^{-\sigma_s(T-\tau)} (1{+}T/2)^{-5/4} \mathrm{d} \tau 
\end{eqnarray*}
and 
where 
\begin{eqnarray*}
H_3(T) & \leq &  C( b_{c,0}^5(T) + b_{c,0}(T)b_{s}(T)  +b_{s}^2(T)), \\
H_4(T) & \leq &  C
( b_{c,\nu^*}(T) b_{c,0}(T) + b_{c,2}(T) b_{c,1}(T) b_{c,0}(T)
+ b_{c,3}(T) b_{c,0}(T)^3 \\ &&+ b_{c,2}(T) b_{c,0}(T)^4 + b_{c,1}(T) b_{c,0}(T)^5 
+  b_{c,0}(T)^7 \\ &&+ b_s(T) b_{c,0}(T)^2+  b_s(T)^2b_{c,0}(T)).
\end{eqnarray*}

\subsection{The final estimates}

We set
$$ 
R(T) = a_{c,0}(T)+b_{c,0}(T)+a_{c,\nu^*}(T)+b_{c,\nu^*}(T)
+  a_{s}(T)+b_{s}(T).
$$
Summing up all estimates yields an inequality 
$$ 
R(T)  \leq R(0) + f(R(T))
$$ 
where $ f $ is at least quadratic in its argument.
Comparing the curves $ R \mapsto R $ and $ R \mapsto \delta+f(R) $, it is  easy to see that 
$ R $ cannot go beyond $ 2 \delta $. Hence,
if $ R(0) < \delta$, with $\delta>0$ sufficiently small, 
especially so small that the implicit function theorem for 
\eqref{mice56} can be applied,
we have the
existence of a $C> 0$ such that $R(T) \leq C$ for all $T \geq 0$.  
Therefore, with this and \eqref{dance2} we are done with the proof of Theorem \ref{yuri1}.
\qed

\paragraph{Acknowledgments.} 
This research was partially supported by the Swiss National Science Foundation grant 171500.

\appendix

\section{The limit profile} \label{appA}

By rescaling $ \varphi $, $ T $, and $ X $ the limit equation  can be brought into 
the form
\[
\partial_T \varphi = - \partial^4_X \varphi - \partial_X \left(
  (\partial_X \varphi)^2\right).
\]
For finding the self-similar solutions we make the ansatz
\[
\varphi (X,T) = \frac{1}{T^{1/4}} \psi \left( \frac{X}{T^{1/4}}
\right) = \frac{1}{T^{1/4}} \psi (\xi).
\]
Using 
$$ 
\partial^n_X \varphi = \ \frac{1}{T^{(n+1)/4}} \psi^{(n)} (\xi )   \qquad \text{and} \qquad 
\partial_T \varphi = - \frac{1}{T^{5/4}} \left( \frac14 \psi +
  \frac14 \xi \psi'\right)
$$
we obtain that $\psi$ satisfies the ODE
\begin{equation} \label{ODE}
0 = - \psi^{(4)} + \frac14 \psi + \frac14 \xi \psi' - \left( (\psi')^2
\right) '.
\end{equation}
We look for solutions homoclinic to the origin, \ie, for solutions which satisfy 
$ \psi(\xi) \to 0 $ for $ |\xi| \to \infty $. In order to do so we first analyze  the linear operator
\[
L \psi = - \psi^{(4)} + \frac14 \xi \psi' + \frac14 \psi,
\]
and then consider 
the nonlinear terms using the implicit function theorem.

For the computation of the spectrum of $L$ we use its
representation in Fourier space, namely 
\[
\frac{1}{2\pi} \int_\mathbb{R} \left( - \psi^{(4)} + \frac14 \xi \psi'
  + \frac14 \psi \right) e^{-ik\xi} \, \mathrm{d} \xi = - k^4 \widehat{\psi} -
  \frac14 k \widehat{\psi}'.
\] 
The eigenvalue problem 
$$
-k^4 \widehat{\psi} - \frac14 k \widehat{\psi}' =
\lambda \widehat{\psi}
$$  
is solved by $\widehat{\psi}_s = k^s e^{-k^4}$
with associated eigenvalue $\lambda_s = - \frac{1}{4} {s}$. It is well known \cite{Wa97} that 
the spectrum depends on the chosen phase space. 
We define
\[
H^n_m = \{ \widehat{u} \in H^n : \left\| \widehat{u} \rho^m\right\|_{H^n} <
\infty\}, 
\qquad 
\text{where}
\quad
\rho (k) = \sqrt{1 + k^2}.
\]
We have $\widehat{\psi} \in H^n_m$ for $\widehat{\psi} = k^s e^{-k^4}$ if $s\in
\mathbb{N}$ or $s > n - \frac12$ and all $m \geq 0$. Hence in $H^n_m$
we have $n$  discrete eigenvalues $\lambda_s = - \frac{1}{4} {s}$ for
$s\in \{ 0,1,\ldots, n-1\}$ and essential spectrum left of $\mathop{\mathrm{Re}}
\lambda = - \frac{1}{4} {n} + \frac18$ due to Sobolev's embedding theorem.

In order to define a projection which separates the eigenspace associated 
to the zero eigenvalue from the rest we consider the 
associated  adjoint operator $ L^\ast $ defined through 
\begin{eqnarray*}
(L\psi, \tilde{\psi})_{L^2} & = & 
{\int_\mathbb{R} \left( - \psi^{(4)} + \frac14 \xi \psi' +
    \frac14 \psi \right) \tilde{\psi} \,\mathrm{d} \xi} \\
&& = \int_\mathbb{R} - \psi'' \tilde{\psi}'' - \frac14 \psi (\xi
\tilde{\psi})' + \frac14 \psi \tilde{\psi} \, \mathrm{d} \xi \\
&& = \int_\mathbb{R} - \psi \tilde{\psi}^{(4)} - \frac14 \psi
\tilde{\psi}' \, \mathrm{d} \xi= (\psi, L^\ast
\tilde{\psi})_{L^2}
\end{eqnarray*}
and so
\[
L^\ast \tilde{\psi} = - \tilde{\psi}^{(4)} - \frac14 \tilde{\psi}'.
\]
It is easy to see that $L^\ast \tilde\psi = 0$ implies $\tilde\psi = const.$.
Therefore, the projection $ P_0 $ on the eigenspace
$\mathop{\mathrm{span}} \{\psi_0\}$ associated to the eigenvalue
$\lambda = 0$ can be defined via the associated adjoint
eigenfunction $\psi^\ast_0 = 1$, \ie,
\[
P_0 u = \langle \psi^\ast_0, u \rangle \psi_0 = \left( \int_\mathbb{R}
  u(\xi) \, \mathrm{d} \xi\right) \cdot \psi_0.
\]
Moreover, let $P_- = I - P_0$. We have $LP_0 = P_0 L$ and $LP_- = P_-
L$.
With these projections we split \eqref{ODE} into two parts. 
We consider $ \psi \in H^2_n $ with $ n \geq 2 $ and set $\psi = A \psi_0 + \psi_-$, with $A \in
\mathbb{R}$ and $P_0 \psi_- = 0$, and obtain 
\begin{eqnarray*}
L(A\psi_0 ) + P_0 \left( - \left( ( \psi')^2\right)'\right) &= &0 ,\\
L\psi_- + P_-\left( - \left( ( \psi')^2\right)'\right) &= &0 .
\end{eqnarray*}
The first equation is satisfied identically, since $L\psi_0 = 0$ and
$$
P_0 \left( - \left( ( \psi')^2\right)'\right) = \left(
  \int_\mathbb{R} - \left( (\psi')^2\right)' \, \mathrm{d}\xi\right)
\psi_0 = 0.
$$
Therefore, we find
\[
\psi_- = - L^{-1} P_- \left( \left( ( A \psi_0 + \psi_-
    )^2\right)'\right). 
\]
For $\left| A\right|$ sufficiently small, the r.h.s.\ is a contraction in $ H^2_n$,
and so we have a unique solution $\psi^*_- (A)\in H^2_n$, resp., $\psi^* (A) = A
\psi_0 + \psi^*_- (A) \in H^2_n$.

\section{Formal irrelevance in the diagonalized system}

\label{appB}

The goal of this section is to provide all calculations necessary for the proof of Lemma \ref{lem33}.
We recall the rules 
$$
V_c \sim T^{-1/4},  \quad \partial_X
\sim T^{-1/4}, \quad \text{and} \quad \partial_T \sim T^{-1}
$$ 
and start now expanding our equations in powers of $ T^{-1/4} $.
In order to keep the notation on a reasonable level 
we abbreviate all terms with $ \mathcal{O}(T^{-\alpha}) $
which turn out to be obviously irrelevant  w.r.t. the linear dynamics.  Herein,
$ \alpha > 0 $ will vary from formula to formula.
For instance a term of power $T^{-3/4} $ must contain one $ V_c $ and two $x$-derivatives,
or $ V_c^2 $ and one $x$-derivative, or $ V_c^3 $.
We could have called this expansion parameter $ \varepsilon $, but we  thought, it is more natural to keep $ T^{-1/4} $ as small expansion parameter.

We recall  the eigenvalues and eigenvectors 
of the operator $ L $ in Fourier space, \ie, of the matrix
$$\left(
\begin{array}{cc}
-k^2- \frac{4}{3} & - 2 \sqrt{\frac13} ik\\ 2 \sqrt{\frac13} ik & -k^2 
\end{array}\right).
$$
The eigenvalues are zeroes of  the characteristic polynomial, \ie, 
$$ 
(k^2+\lambda)^2 +  \frac{4}{3}(k^2+\lambda) - \frac{4}{3} k^2= 0.
$$
The eigenvalues are then given
$$
\lambda_{1/2}(k) = -\frac23\pm \frac23\sqrt{1+3k^2} = \left\{ \begin{array}{cl}
-\frac34 k^4 +\mathcal{O}(k^6), & \text{for } k=1, \\ 
- \frac{4}{3} - 2 k^2 +\frac34 k^4 +\mathcal{O}(k^6), & \text{for } k=2 .
\end{array} \right.
$$
For the change of variables  leading to the diagonalization we need to 
compute the associated eigenvectors $ \widehat{\varphi}_1 $ and $ \widehat{\varphi}_2 $.
For our purposes it is sufficient to compute an expansion of the eigenfunctions 
at $ k = 0 $. In order to keep the following calculations on a reasonable level,
we use a slightly different normalization. We set the second component of $ \widehat{\varphi}_1 $ 
and the first  component of $ \widehat{\varphi}_2 $ to one.

{\bf i)} We start with $ \lambda_1 $.
We have to find the kernel of the matrix 
 $$\left(
\begin{array}{cc}
-k^2- \frac{4}{3}  + \frac34 k^4 +\mathcal{O}(k^6)& - 2 \sqrt{\frac13} ik\\ 2 \sqrt{\frac13} ik & -k^2 
+ \frac34 k^4 +\mathcal{O}(k^6)
\end{array}\right) = A_0 + k A_1 + k^2 A_2 + k^4 A_4 +\mathcal{O}(k^6),
$$
with 
$$
A_0 =
\left(
\begin{array}{cc}
 - \frac{4}{3}  &0 \\ 0&0 \end{array}\right) ,
\quad
A_1 = \left(
\begin{array}{cc}
0& - 2 \sqrt{\frac13} i\\ 2 \sqrt{\frac13} i &0\end{array}\right) ,
$$
and 
$$
A_2 =
\left(
\begin{array}{cc}
 - 1  &0 \\ 0&-1 \end{array}\right) ,
\quad
A_4 =
\left(
\begin{array}{cc}
  \frac34  &0 \\ 0&\frac34 \end{array}\right).
$$
It turns out that for the associated eigenvector it is sufficient  to make the ansatz
$$ 
 \widehat{\varphi}_1  = \left(
\begin{array}{c}
 a_1k + a_3 k^3 +  \mathcal{O}(k^5)\\ 1 \end{array}\right) .
$$ 
At $ k^0 $ we find $ A_0  \left(
\begin{array}{c}
 0\\ 1 \end{array}\right) =  \left(\begin{array}{c}
 0\\ 0 \end{array}\right) $ which is satisfied.
 
 At $ k^1 $ we find 
 $$ 
 A_0  \left(
\begin{array}{c}
 a_1\\ 0 \end{array}\right)  + A_1 \left(
\begin{array}{c}
 0\\ 1 \end{array}\right) =  \left(\begin{array}{c}
 0\\ 0 \end{array}\right) 
 $$ 
 which leads to $ -\frac43 a_1 = 2 \sqrt{\frac13} i $ or equivalently 
 to $ a_1 = - \frac{\sqrt{3}}{2} i $.
 
 At $ k^2 $ we find 
 $$ 
 A_1  \left(
\begin{array}{c}
 a_1\\ 0 \end{array}\right)  + A_2 \left(
\begin{array}{c}
 0\\ 1 \end{array}\right) =  \left(\begin{array}{c}
 0\\ 0 \end{array}\right) 
 $$ 
which is satisfied.

At $ k^3 $ we find 
 $$ 
 A_0  \left(
\begin{array}{c}
 a_3\\ 0 \end{array}\right)  
 + A_2 \left(
\begin{array}{c}
 a_1\\ 0 \end{array}\right) =  \left(\begin{array}{c}
 0\\ 0 \end{array}\right) 
 $$ 
which leads 
 to $ a_3 =  \frac34 \frac{\sqrt{3}}{2} i $.

At $ k^4 $ we find 
 $$ 
 A_1  \left(
\begin{array}{c}
 a_3\\ 0 \end{array}\right)  + A_4 \left(
\begin{array}{c}
 0\\ 1 \end{array}\right) =  \left(\begin{array}{c}
 0\\ 0 \end{array}\right) 
 $$ 
which is satisfied.
Therefore, we found
$$ 
 \widehat{\varphi}_1  = \left(
\begin{array}{c}
  - \frac{\sqrt{3}}{2} i k + \frac34 \frac{\sqrt{3}}{2} i k^3 +  \mathcal{O}(k^5)\\ 1 \end{array}\right) .
$$ 

{\bf ii)} Next we come to $ \lambda_2 $.
We have to find the kernel of the matrix 
 $$\left(
\begin{array}{cc}
k^2-   \frac34 k^4 +\mathcal{O}(k^6)& - 2 \sqrt{\frac13} ik\\ 2 \sqrt{\frac13} ik &
\frac43
 +k^2 
- \frac34 k^4 +\mathcal{O}(k^6)
\end{array}\right) = B_0 + k B_1 + k^2 B_2 + k^4 B_4 +\mathcal{O}(k^6),
$$
with 
$$
B_0 =
\left(
\begin{array}{cc}
0&0 \\0&\frac{4}{3}   \end{array}\right) ,
\quad
B_1 = \left(
\begin{array}{cc}
0& - 2 \sqrt{\frac13} i\\ 2 \sqrt{\frac13} i &0\end{array}\right) ,
$$
and 
$$
B_2 =
\left(
\begin{array}{cc}
 1  &0 \\ 0&1 \end{array}\right) ,
\quad
B_4 =
\left(
\begin{array}{cc}
 - \frac34  &0 \\ 0&- \frac34 \end{array}\right).
$$
It turns out that for the associated eigenvector it is sufficient  to make the ansatz
$$ 
 \widehat{\varphi}_2  = \left(
\begin{array}{c} 1 \\
 b_1k + b_3 k^3 +  \mathcal{O}(k^5) \end{array}\right) .
$$ 
At $ k^0 $ we find $ B_0  \left(
\begin{array}{c}
 1\\ 0 \end{array}\right) =  \left(\begin{array}{c}
 0\\ 0 \end{array}\right) $ which is satisfied.
 
 At $ k^1 $ we find 
 $$ 
 B_0  \left(
\begin{array}{c}
 0\\ b_1 \end{array}\right)  + B_1 \left(
\begin{array}{c}
 1\\ 0 \end{array}\right) =  \left(\begin{array}{c}
 0\\ 0 \end{array}\right) 
 $$ 
 which leads to $ \frac43 b_1 = - 2 \sqrt{\frac13} i $ or equivalently 
 to $ b_1 = - \frac{\sqrt{3}}{2} i $.
 
 At $ k^2 $ we find 
 $$ 
 B_1  \left(
\begin{array}{c}
0 \\  b_1 \end{array}\right)  + B_2 \left(
\begin{array}{c}
 1\\ 0 \end{array}\right) =  \left(\begin{array}{c}
 0\\ 0 \end{array}\right) 
 $$ 
which is satisfied.

At $ k^3 $ we find 
 $$ 
 B_0  \left(
\begin{array}{c}
0 \\ b_3  \end{array}\right)  
 + B_2 \left(
\begin{array}{c}
 0 \\ b_1 \end{array}\right) =  \left(\begin{array}{c}
 0 \\ 0 \end{array}\right) 
 $$ 
which leads 
 to $ b_3 = - \frac34 \frac{\sqrt{3}}{2} i $.

At $ k^4 $ we find 
 $$ 
 B_1  \left(
\begin{array}{c}
 0 \\ b_3 \end{array}\right)  + B_4 \left(
\begin{array}{c}
 1\\ 0 \end{array}\right) =  \left(\begin{array}{c}
 0\\ 0 \end{array}\right) 
 $$ 
which is satisfied.
Therefore, we found
$$ 
 \widehat{\varphi}_2 = \left(
\begin{array}{c} 1 \\ 
  - \frac{\sqrt{3}}{2} i k +\frac34 \frac{\sqrt{3}}{2} i k^3 +  \mathcal{O}(k^5) \end{array}\right) .
$$ 
We use these eigenfunctions to diagonalize 
$$
\partial_T \widehat{v} =\widehat{L} \widehat{v} + \widehat{N}(\widehat{v}), 
$$
with $ \widehat{v} = \widehat{S} \widehat{\widetilde{v}} $ with matrix $ \widehat{S}(k) = ( \widehat{\varphi}_2(k) \quad   \widehat{\varphi}_1(k)) $ to obtain
$$
\partial_T \widehat{\widetilde{v}} =\widehat{\Lambda} \widehat{\widetilde{v}} + \widehat{S}^{-1} \widehat{N}(\widehat{S} \widehat{\widetilde{v}}), 
$$
with $ \Lambda= {\rm diag}(\lambda_2,\lambda_1) $.
Again our purposes it is sufficient to compute an expansion of $ \widehat{S} $ and $ \widehat{S}^{-1} $
at $ k = 0 $. We find 
$$ 
\widehat{S}(k) = 
 \left(
\begin{array}{cc} 1 &  - \frac{\sqrt{3}}{2} i k +\frac34 \frac{\sqrt{3}}{2} i k^3 +  \mathcal{O}(k^5) 
 \\ 
  - \frac{\sqrt{3}}{2} i k +\frac34 \frac{\sqrt{3}}{2} i k^3 +  \mathcal{O}(k^5) 
  & 1 \end{array}\right).
 $$
 We compute 
 $$
 \det= 1 + \frac34 k^2-  \frac98 k^4+ \mathcal{O}(k^6) ,
 $$
 and so
 \begin{eqnarray*}
\widehat{S}^{-1}(k) & = & \frac{1}{1 + \frac34 k^2-  \frac98 k^4+  \mathcal{O}(k^6) }
 \left(
\begin{array}{cc} 1 &   \frac{\sqrt{3}}{2} i k -\frac34 \frac{\sqrt{3}}{2} i k^3 +  \mathcal{O}(k^5) 
 \\ 
  \frac{\sqrt{3}}{2} i k -\frac34 \frac{\sqrt{3}}{2} i k^3 +  \mathcal{O}(k^5) 
  & 1 \end{array}\right)
  \\
  & = &  \left(
\begin{array}{cc} 1 -  \frac34 k^2 + \frac{27}{16} k^4&   \frac{\sqrt{3}}{2} i k -\frac32 \frac{\sqrt{3}}{2} i k^3  \\ 
  \frac{\sqrt{3}}{2} i k -\frac32 \frac{\sqrt{3}}{2} i k^3 
  & 1-  \frac34 k^2 + \frac{27}{16} k^4 \end{array}\right)+  \mathcal{O}(k^5).
 \end{eqnarray*}
  In order to calculate  the diagonalized system for 
  $ \widetilde{v} = (V_s,V_c) $,
  we start with  the non-diagonalized system for $ v= (V_r,V_i) $, namely 
 $$
 \underbrace{\partial_T V_r}_{\sim T^{-3/2}} = g_r, \qquad \underbrace{\partial_T V_i}_{\sim T^{-5/4}} = g_i,
 $$ 
 where
 \begin{eqnarray*}
g_r & = & \underbrace{\partial_X^2 V_r}_{\sim T^{-1}} -\underbrace{\tfrac{4}{3} V_r}_{\sim T^{-1/2}}  -\underbrace{2\sqrt{\tfrac{1}{3}}\partial_X V_i}_{\sim T^{-1/2}}
- \frac{2}{3}(\underbrace{3 V_r^2}_{\sim T^{-1}} + \underbrace{V_i^2}_{\sim T^{-1/2}}+\underbrace{V_r^3}_{\sim T^{-3/2}}+\underbrace{V_rV_i^2}_{\sim T^{-1}}), \\
g_i  & = &  \underbrace{\partial_X^2 V_i}_{\sim T^{-3/4}}  +  \underbrace{2\sqrt{\tfrac{1}{3}}\partial_X  V_r}_{\sim T^{-3/4}}   -\frac{2}{3}(\underbrace{2V_rV_i}_{\sim T^{-3/4}} +\underbrace{V_r^2V_i}_{\sim T^{-5/4}} +\underbrace{V_i^3}_{\sim T^{-3/4}} ).
\end{eqnarray*}
We compute 
$$
 \underbrace{\partial_T V_s}_{\sim T^{-3/2}} = g_s, \qquad \underbrace{\partial_T V_c}_{\sim T^{-5/4}} = g_c.
 $$ 
In order to avoid working with the convolutions in Fourier space we consider 
$ \widehat{S} $ and $ \widehat{S}^{-1} $ in physical space. We obtain
$$ 
S(\partial_X) = 
 \left(
\begin{array}{cc} 1 &  - \frac{\sqrt{3}}{2} \partial_X -\frac34 \frac{\sqrt{3}}{2} \partial_X^3  \\ 
  - \frac{\sqrt{3}}{2} \partial_X -\frac34 \frac{\sqrt{3}}{2} \partial_X^3 
  & 1 \end{array}\right)+  \mathcal{O}(T^{-5/4}) 
 $$
and 
$$
S^{-1}(\partial_X)    =    \left(
\begin{array}{cc} 1 +  \frac34 \partial_X^2 + \frac{27}{16} \partial_X^4&   \frac{\sqrt{3}}{2} \partial_X +\frac32 \frac{\sqrt{3}}{2} \partial_X^3  \\ 
  \frac{\sqrt{3}}{2} \partial_X +\frac32 \frac{\sqrt{3}}{2} \partial_X^3 
  & 1+  \frac34 \partial_X^2 + \frac{27}{16} \partial_X^4 \end{array}\right)
+ \mathcal{O}(T^{-5/4}) .
$$
In the following lengthy calculations, in $ g_r $ and $ g_s $ we have to keep terms 
of order $ \mathcal{O}(T^{-1/2}) $ and $ \mathcal{O}(T^{-1}) $, and
in $ g_i $ and $ g_c $ we have to keep terms 
of order $ \mathcal{O}(T^{-3/4}) $ and $ \mathcal{O}(T^{-5/4}) $.
With
$$
V_r = \underbrace{V_s}_{\sim T^{-1/2}} - \frac{\sqrt{3}}{2} \underbrace{\partial_X V_c}_{\sim T^{-1/2}} -\frac34 \frac{\sqrt{3}}{2} \underbrace{\partial_X^3 V_c}_{\sim T^{-1}}
$$
and 
$$
V_i = \underbrace{V_c}_{\sim T^{-1/4}}  - \frac{\sqrt{3}}{2}\underbrace{\partial_X V_s}_{\sim T^{-3/4}}-\frac34 \frac{\sqrt{3}}{2} \underbrace{\partial_X^3 V_s}_{\sim T^{-5/4}}.
$$
After another lengthy calculation  we arrive at 
$$ 
g_s = s_2 + s_4 +  \mathcal{O}(T^{-3/2}) \qquad \text{and} \qquad 
g_c = s_3 + s_5 +  \mathcal{O}(T^{-7/4}) ,
$$ 
where
\begin{eqnarray*}
 s_2  & = & - \frac43 V_s - \frac23 V_c^2 ,\\ 
 s_4 & = &  2 \partial_X^2 V_s + \frac16\left(-4 V_c^2 V_s -12 V_s^2 - 4 \sqrt{3} V_c^2 \partial_X V_c
 + 8  \sqrt{3} V_s \partial_X V_c - 9 (\partial_X V_c)^2 \right) ,\\ 
 s_3  & = & -\frac23 (V_c^3+ 2 V_c V_s), \\
 s_5 & = & - \frac34 \partial_X^4 V_c + \frac16(- 4 V_c V_s^2-15 V_c (\partial_X V_c)^2
 +4 \sqrt{3} V_c^2 \partial_X V_s- 8  \sqrt{3} V_s \partial_X V_s \\ && \qquad + 6 (\partial_X V_c)\partial_X V_s
 - 6 V_c^2  \partial_X^2 V_c+ 12 V_s  \partial_X^2 V_c -18 \sqrt{3}  (\partial_X V_c)  \partial_X^2 V_c).
 \end{eqnarray*}
Putting in $ g_s $ the terms of order $ \mathcal{O}(T^{-1/2}) $ to zero, \ie,  $ s_2 = 0 $, 
yields $ V_s = - \frac12 V_c^2 $.
Inserting this in the terms of order   $ \mathcal{O}(T^{-3/4}) $ in $ g_c $ yields 
$ s_3 = 0 $, \ie, these terms vanish identically. 
The next order correction  of $ V_s $ will influence  the terms of 
$ \mathcal{O}(T^{-5/4}) $ in $ g_c $ and so we compute the transform \eqref{coo22}
completely.
We introduce $ \widehat{V}_s^*(\widehat{W}_c) $ by
$ \widehat{v}_s^* (\widehat{w}_c) = \widehat{V}_s^*(\widehat{W}_c) \, \widehat{\varphi}_1$,
so that $ V_s = V_s^*(W_c) + W_s $. We have
\begin{eqnarray*}
 V_s^*(W_c) & = &  - \frac12 W_c^2 + \frac32 \partial_X^2 W_s  \\ && + \frac18\left(-4W_c^2 W_s -12 W_s^2 - 4 \sqrt{3} W_c^2 \partial_X W_c
 + 8  \sqrt{3} W_s \partial_X W_c - 9 (\partial_X W_c)^2 \right) + \mathcal{O}(T^{-3/2})  \\ 
 & = &  - \frac12 W_c^2 - \frac34 \partial_X^2 (W_c^2)  \\ && + \frac18\left(2 W_c^4 - 3 W_c^4 - 4 \sqrt{3} W_c^2 \partial_X W_c
 - 4  \sqrt{3} W_c^2 \partial_X W_c - 9 (\partial_X W_c)^2 \right)+ \mathcal{O}(T^{-3/2}) \\ 
  & = &  - \frac12 W_c^2 - \frac34 \partial_X^2 (W_c^2)  + \frac18\left(- W_c^4 - 8 \sqrt{3} W_c^2 \partial_X W_c
  - 9 (\partial_X W_c)^2 \right)+ \mathcal{O}(T^{-3/2}) .
\end{eqnarray*}
Inserting $ V_s = V_s^*(W_c) + W_s $ into $ s_5 $ gives a big number of cancellations
and so we finally obtain 
$$ s_5 = - \frac34 \partial_X^4 W_c -\frac32 \sqrt{3} \partial_X((\partial_X W_c )^2).$$
Thus, the first equation of \eqref{shs1} is of the form 
$$ 
\partial_T W_c = - \frac34 \partial^4_X W_c - \frac32
\sqrt{3} \partial_X \left(( \partial_X W_c)^2\right) + \check{B}_2(W_c)
+ \check{B}_3(W_c)+ \check{B}_4(W_c)+ \check{B}_5(W_c) + \check{g}_c(W_c,W_s)
$$ 
where
$$
\check{B}_2(W_c)
+ \check{B}_3(W_c)+ \check{B}_4(W_c)+ \check{B}_5(W_c) +  \check{g}_c(W_c,W_s)
$$
decays at least with a rate  $ T^{-3/2} $, with $ \check{B}_m $ standing for the 
$ m $-linear terms
in $ W_c $.
In Fourier space they can be written as 
\begin{eqnarray*}
 \check{B}_{2}(\widehat{W}_c)(k)   & = &    \int {K}_2(k,k-l,l)  \widehat{W}_c(k-l) \widehat{W}_c(l) \mathrm{d}l,
   \\
 \check{B}_{3}(\widehat{W}_c)(k)  & = & \int {K}_3(k,k-l,l-l_1,l_1)  \widehat{W}_c(k-l) \widehat{W}_c(l-l_1) \widehat{W}_c(l_1)\mathrm{d}l_1\mathrm{d}l ,
\end{eqnarray*}
and similarly for $  \check{B}_{4} $ and $  \check{B}_{5} $. 
Since the decay rates in time correspond one-to-one to the powers w.r.t.~$ W_c $ or to 
the decay rates of the kernels at the origin, we necessarily have 
\begin{equation} 
\begin{array}{rcl}
\label{alcu5bg}
|{K}_2(k,k_1,k_2) | & \leq & C (|k|^4+|k_1|^4+|k_2|^4), \\
|{K}_3(k,k_1,k_2,k_3) | & \leq & C (|k|^3+|k_1|^3+|k_2|^3+|k_3|^3), \\
|{K}_4(k,k_1,k_2,k_3,k_4) | & \leq & C (|k|^2+|k_1|^2+|k_2|^2+|k_3|^2+ |k_4|^2), \\
|{K}_5(k,k_1,k_2,k_3,k_4,k_5) | & \leq & C (|k|+|k_1|+|k_2|+|k_3|+ |k_4|+|k_5|),
\end{array}
\end{equation}
for $  k, k_1,k_2,k_3,k_4,k_5 \to 0 $.
With the same argument the statement about  $ M_2 $ and $ K^* $ follows.

\bibliography{biblio}
\bibliographystyle{halpha}

\end{document}